\documentclass [11pt]{amsart}
\usepackage {amsmath, amssymb, amscd, mathrsfs, url,pinlabel,hyperref, verbatim}
\usepackage[margin=1.25in,centering,letterpaper,dvips]{geometry}
\usepackage{color,multirow,dcpic,latexsym,pictexwd,graphicx,epstopdf,comment}

\newtheorem {theorem}{Theorem}
\newtheorem {lemma}[theorem]{Lemma}
\newtheorem {proposition}[theorem]{Proposition}
\newtheorem {corollary}[theorem]{Corollary}
\newtheorem {conjecture}[theorem]{Conjecture}
\newtheorem {definition}[theorem]{Definition}

\theoremstyle{remark}

\newtheorem {remark}[theorem]{Remark}
\newtheorem {example}[theorem]{Example}

\numberwithin{equation}{section}
\numberwithin{theorem}{section}

\newcommand{\maxtb}{\overline{tb}}
\newcommand{\maxsl}{\overline{sl}}
\newcommand{\rank}{\operatorname{rank}}
\newcommand{\ZZ}{\mathbb{Z}}
\newcommand{\Z}{\mathbb{Z}}
\newcommand{\R}{\mathbb{R}}

\newcommand{\Q}{\mathbb{Q}}

\newcommand{\rv}{\vec{r}}
\newcommand{\sv}{\vec{s}}
\newcommand{\xican}{\xi_{\mathrm{can}}}

\newcommand{\sh}{S\mathbb{H}}
\newcommand{\lhho}{L\mathbb{H}^{\mathrm{Ho}}}

\newcommand{\spc}{\mathrm{Spin}^c}
\newcommand{\hfhat}{\widehat{HF}}
\newcommand{\hfp}{HF^{+}}

\specialcomment{steven}{%
  \begingroup\color{blue}  Steven \; $\blacktriangleright$ \;}{%
  \endgroup}
\specialcomment{tye}{%
  \begingroup\color{red} Tye \; $\blacktriangleright$ \;}{%
  \endgroup}

\title{Contact structures and reducible surgeries}
\date{}

\author{Tye Lidman}
\address{Department of Mathematics, The University of Texas at Austin}
\email{tlid@math.utexas.edu}

\author{Steven Sivek}
\address{Department of Mathematics, Princeton University}
\email{ssivek@math.princeton.edu}

\begin{document}

\begin{abstract}
We apply results from both contact topology and exceptional surgery theory to study when Legendrian surgery on a knot yields a reducible manifold.  As an application, we show that a reducible surgery on a non-cabled positive knot of genus $g$ must have slope $2g-1$, leading to a proof of the cabling conjecture for positive knots of genus 2.  Our techniques also produce bounds on the maximum Thurston-Bennequin numbers of cables.
\end{abstract}

\maketitle


\section{Introduction}
\subsection{Background}
Given a knot in $S^3$, an important problem in three-manifold topology is to classify the Dehn surgeries on $K$.  One of the biggest open problems in Dehn surgery is to determine the knots which admit reducible surgeries.  Gabai's proof of Property R \cite{gabai} shows that if 0-surgery on $K$ is reducible, then $K$ is in fact the unknot.  In particular, 0-surgery on a knot is always prime.  However, many non-trivial knots do have reducible surgeries.  If $K$ is the $(p,q)$-cable of a knot $K'$ (where $p$ is the longitudinal winding) and $U$ is the unknot, then $S^3_{pq}(K) = S^3_{p/q}(U) \# S^3_{q/p}(K')$.\footnote{The manifold $S^3_{p/q}(U)$ is of course a lens space, but we write it this way for now to avoid confusion: it is often called $L(p,q)$ by $3$-manifold topologists but $-L(p,q)$ by contact geometers.  We will use the latter convention throughout this paper.}  Conjecturally, these are the only such examples.  
\begin{conjecture}[Cabling Conjecture, Gonzalez-Acu\~na--Short \cite{gonzalez-acuna-short}]
Suppose Dehn surgery on a non-trivial knot $K$ is reducible.  Then $K = C_{p,q}(K')$ for some $K'$ and the surgery coefficient is $pq$.  
\end{conjecture}
The cabling conjecture is known for torus knots \cite{moser} and satellite knots \cite{scharlemann}, but is still open for hyperbolic knots.  Two key observations for the reducible surgeries on cables are that the surgery always produces a lens space summand and the surgery coefficients are integral.  In fact, Gordon and Luecke showed that both of these conditions must hold for a reducible surgery on any non-trivial knot.
\begin{theorem}[Gordon--Luecke \cite{gordon-luecke-integral,gordon-luecke-lens}]
If some Dehn surgery $S^3_r(K)$ on a non-trivial knot $K$ is reducible, then $r\in\ZZ$ and the surgery contains a lens space summand.
\end{theorem}
One consequence of this is that if $n$-surgery on $K$ is reducible, then $|n| \geq 2$.  Another consequence is that a reducible surgery on a cable knot will have exactly two summands.  It is not known that a reducible surgery on a non-cable knot cannot have more than two summands; however, it is known that if there are not two summands, the reducible manifold is a connected sum of two lens spaces and an irreducible homology sphere \cite{howie}.  A weaker version of the cabling conjecture is the three summands conjecture, which says that reducible surgery never has more than two summands.  

In this paper, it is our goal to study when Legendrian surgery on a knot can be reducible.  Recall that for a Legendrian representative of $K$, performing {\em Legendrian surgery} on $K$ is topologically Dehn surgery with coefficient $tb(K) - 1$.  Let $\maxtb(K)$ denote the maximum Thurston-Bennequin number of any Legendrian representative of $K$.  Since stabilizations reduce $tb$ by one, any integral surgery coefficient strictly less than $\maxtb(K)$ will correspond to a Legendrian surgery.  

Our results will all stem from the following theorem of Eliashberg.
\begin{theorem}[Eliashberg \cite{eliashberg-filling, cieliebak-eliashberg}]
\label{thm:steinboundary-intro}
Suppose that $(X,J)$ is a Stein filling of a non-prime contact $3$-manifold $(Y_1,\xi_1) \# (Y_2,\xi_2)$.  Then $(X,J)$ decomposes as a boundary sum $(X_1,J_1) \natural (X_2,J_2)$, where $(X_i,J_i)$ is a Stein filling of $(Y_i,\xi_i)$.\end{theorem}

\subsection{Reducible surgeries for knots with $\maxtb(K) \geq 0$} 
We will first use Theorem~\ref{thm:steinboundary-intro} to prove the following.  
\begin{proposition}\label{prop:decomp}
Let $K$ be a knot in $S^3$ and suppose that $S^3_n(K) = L(p,q) \# Y$ where $n < \maxtb(K)$.  Then:
\begin{enumerate} 
\item $p = |n|$, and $n < -1$;
\item $L(p,q)$ admits a simply-connected Stein filling with intersection form $\langle n \rangle = \langle -p\rangle$;
\item $Y$ is an irreducible integer homology sphere which admits a contractible Stein filling.
\end{enumerate}
\end{proposition}

The three summands conjecture follows immediately for $S^3_n(K)$ when $n < \maxtb(K)$.  If $S^3_n(K)$ has at least three summands then so does $-S^3_n(K) = S^3_{-n}(\overline{K})$, where $\overline{K}$ is the mirror of $K$, so we conclude:
\begin{corollary}
Let $K$ be a knot in $S^3$.  If $S^3_n(K)$ has more than two summands, then $\maxtb(K) \leq n \leq -\maxtb(\overline{K})$.
\end{corollary}

From Proposition~\ref{prop:decomp}, we are able to apply the existence of the Stein fillings to study the cabling conjecture via known results in contact topology, such as the classification of tight contact structures on lens spaces.  For instance, we prove the following.  

\begin{theorem}\label{thm:main}
Let $K$ be a knot in $S^3$ and suppose that $\maxtb(K) \geq 0$.  Then any surgery on $K$ with coefficient less than $\maxtb(K)$ is irreducible.
\end{theorem}

It is a theorem of Matignon and Sayari \cite{matignon-sayari} that if $S^3_n(K)$ is reducible for a non-cable $K$, then $|n| \leq 2g(K) - 1$, where $g(K)$ is the Seifert genus of $K$.  Therefore, if $\maxtb(K)$ is large, this can strongly restrict the range of possible reducible surgeries on $K$.  We illustrate this with positive knots. 
\begin{theorem}\label{thm:pos}
Suppose that $K$ is a non-trivial positive knot which is not a cable.  If $S^3_n(K)$ is reducible, then $n = 2g(K) - 1$.  Consequently, there are no essential, punctured projective planes in the complement of $K$.
\end{theorem}  

Without additional information, one cannot apply Theorem~\ref{thm:main} to rule out the case of $(2g(K) - 1)$-surgery, since Bennequin's inequality \cite{bennequin} implies that $\maxtb(K) \leq 2g(K) - 1$.  However, in some cases, one can in fact rule out this final surgery coefficient.  
\begin{theorem}\label{thm:cabling-genus2-positive}
The cabling conjecture holds for genus 2 positive knots.  
\end{theorem}

\begin{remark}
In practice, for most knots $\maxtb(K)$ is negative and thus Theorem~\ref{thm:main} does not apply.  However, large classes of knots do have $\maxtb(K) \geq 0$, such as strongly quasipositive knots \cite{rudolph-sliceness}, and so this shows that strongly quasipositive knots (among others) do not have negative reducible surgeries.
\end{remark}

Observe that in Theorem~\ref{thm:main}, we do not require that $K$ be a non-cable.  Further, since the maximum Thurston-Bennequin number of the unknot is $-1$, we do not need a non-triviality assumption either.  In light of Theorem~\ref{thm:main}, we make the following conjecture.

\begin{conjecture}\label{conj:legcablingconj}
Legendrian surgery on a knot in the tight contact structure on $S^3$ is never reducible.  
\end{conjecture}

\subsection{Knots with $\maxtb(K) < 0$}
While we are not able to prove Conjecture~\ref{conj:legcablingconj} for knots with $\maxtb(K) < 0$, we are able to establish some partial results such as the following.

\begin{theorem}
\label{thm:tb-negative-thm}
Let $K$ be a knot in $S^3$ with $\maxtb(K) < 0$.  If $S^3_n(K)$ is reducible for some $n < \maxtb(K)$, and $W$ is the trace of this surgery, then at least one of the following must hold:
\begin{enumerate}
\item \label{i:tb-negative-lp1} $S^3_n(K) = S^3_n(U) \# Y$.  If this is the case then $W$ is necessarily diffeomorphic to $D_n \natural Z$, where $D_n$ is the disk bundle over $S^2$ with Euler number $n$.

\item \label{i:tb-negative-lp4} $\maxtb(K)=-6$, $n=-7$, and $S^3_{-7}(K) = S^3_{-7}(T_{2,-3}) \# Y$ where $T_{2,-3}$ is the left-handed trefoil.  Moreover, $W$ is diffeomorphic to $X \natural Z$ where $X$ is the trace of $-7$-surgery on $T_{2,-3}$.

\item \label{i:tb-negative-general} $n \geq 4\lfloor\frac{\maxtb(K)}{2}\rfloor + 6$.
\end{enumerate}
In each of the first two cases, $Y$ is an irreducible homology sphere bounding the contractible Stein manifold $Z$.
\end{theorem}
\begin{remark}
In case~\eqref{i:tb-negative-lp4} above, we recall that Moser \cite{moser} showed that $S^3_{-7}(T_{2,-3})$ is in fact the lens space $S^3_{-7/4}(U)$.
\end{remark}

\begin{remark}
If the trace $W$ of a reducible $n$-surgery on $K$ has the form $D_{n} \natural Z$, then the generator of $H_2(W)\cong \ZZ$ is represented by a smoothly embedded sphere even if $K$ is not smoothly slice.  
\end{remark}

\begin{corollary}
\label{cor:tb-at-least--4}
If $-8 \leq \maxtb(K) \leq -1$, then any reducible surgery on $K$ with coefficient $n < \maxtb(K)$ has the form $S^3_n(K) = S^3_n(U) \# Y$, except possibly when $\maxtb(K)=-6$ and $S^3_{-7}(K) = S^3_{-7}(T_{2,-3}) \# Y$.  In both cases $Y$ is an irreducible homology sphere which bounds a contractible Stein manifold.
\end{corollary}

\begin{proof}
Suppose that the lens space summand of the reducible surgery is not $S^3_n(U)$.  Theorem~\ref{thm:tb-negative-thm} says that either $\maxtb(K)=-6$ and $n=-7$, or since $\maxtb(K) \geq -8$ we have $n \geq -10$.  We will see (Remark~\ref{rem:small-lens-spaces}) that this forces the lens space summand $L(|n|,q)$ to be $S^3_{-7}(T_{2,-3}) \cong S^3_{-7/4}(U)$.  If this lens space arises from case~\eqref{i:tb-negative-general} of Theorem~\ref{thm:tb-negative-thm}, then we have $n=-7 \geq 4\lfloor\frac{\maxtb(K)}{2}\rfloor + 6$, hence $\maxtb(K) \leq -7$, contradicting the assumption that $n < \maxtb(K)$.  Thus it can only arise from case~\eqref{i:tb-negative-lp4}, in which case $\maxtb(K)=-6$.
\end{proof}


We cannot guarantee that there do not exist negative reducible surgeries of slope at least $\maxtb(K)$ which satisfy one of the conclusions of Theorem~\ref{thm:tb-negative-thm}: for example, if $K$ is the $(2,-1)$-cable of the right handed trefoil $T_{2,3}$ then $\maxtb(K)=-2$ and $S^3_{-2}(K) = S^3_{-2}(U) \# S^3_{-1/2}(T_{2,3}) = S^3_{-2}(U) \# \Sigma(2,3,13)$.  Note that $\Sigma(2,3,13)$ even bounds a smoothly contractible 4-manifold, as shown by Akbulut and Kirby \cite{akbulut-kirby}.

\begin{corollary}
\label{cor:tb-negative-1-arf}
Suppose that $\frac{\Delta''_K(1)}{2}$ is odd and $\maxtb(K) < 0$.  If $S^3_n(K)$ is reducible for some $n < \maxtb(K)$, then either $(\maxtb(K),n)=(-6,-7)$ or $n \geq 4\lfloor\frac{\maxtb(K)}{2}\rfloor + 6$.
\end{corollary}

\begin{proof}
Suppose $S^3_n(K)$ is a reducible Legendrian surgery and $n < 4\lfloor\frac{\maxtb(K)}{2}\rfloor + 6$ but $(\maxtb(K),n) \neq (-6,-7)$. Theorem~\ref{thm:tb-negative-thm} says that $S^3_n(K) = S^3_n(U) \# Y$, where $Y$ is a homology sphere which bounds a contractible Stein manifold.  The surgery formula for the Casson-Walker invariant $\lambda$, as stated by Boyer--Lines \cite{boyer-lines} (see also Walker \cite{walker}), implies that
\[ \lambda(S^3_n(K)) - \lambda(S^3_n(U)) = \frac{1}{n}\frac{\Delta''_K(1)}{2}. \]
The Casson-Walker invariant is additive under connected sums with homology spheres, so the left side is equal to $\lambda(Y)$, which is an even integer since $Y$ bounds a smoothly contractible manifold and thus has vanishing Rokhlin invariant \cite{akbulut-mccarthy}.  We conclude that $\frac{\Delta''_K(1)}{2} \in 2n\ZZ$, which is impossible since $\frac{\Delta''_K(1)}{2}$ is odd by assumption.
\end{proof}

\begin{remark}
The requirement that $n<\maxtb(K)$ is necessary in order to rule out $S^3_n(U)$ summands: if $K$ is the $(3,-1)$-cable of the right handed trefoil $T_{2,3}$, then $\maxtb(K) = -3$ by \cite[Theorem~1.7]{etnyre-lafountain-tosun}, and $S^3_{-3}(K) = S^3_{-3}(U)\# S^3_{-1/3}(T_{2,3})$ is reducible but $\frac{\Delta''_K(1)}{2} = 9$ is odd.  (In this case we would have $Y= \Sigma(2,3,19)$, and so $\lambda(Y)$ is odd.)
\end{remark}

\subsection{Maximum Thurston-Bennequin numbers for cables}
Combining Theorem~\ref{thm:main} with the fact that cables have reducible surgeries, we are also able to say something about the maximum Thurston-Bennequin numbers of cables, cf.\ \cite{etnyre-honda-cabling, etnyre-lafountain-tosun, tosun}; this technique was originally used by Etnyre--Honda \cite[Lemma~4.9]{etnyre-honda-knots} to compute  $\maxtb$ for negative torus knots.  Let $C_{p,q}(K)$ denote the $(p,q)$-cable of $K$, and note that for nontrivial cables we can assume that $p \geq 2$ since $C_{p,q}(K)=C_{-p,-q}(K)$ up to orientation.

\begin{corollary}\label{cor:maxtbcables}
Suppose that $p\geq 2$ and $\gcd(p,q)=1$, and assume that $q\neq -1$.
\begin{itemize}
\item If $q < p\cdot\maxtb(K)$, then $\maxtb(C_{p,q}(K)) = pq$.
\item If $q > p\cdot\maxtb(K)$, then $pq - (q - p\cdot \maxtb(K)) \leq \maxtb(C_{p,q}(K)) \leq pq$.
\end{itemize}
\end{corollary}

\begin{proof}
Letting $K'=C_{p,q}(K)$, we first prove that $\maxtb(K') \leq pq$.  We suppose for contradiction that $\maxtb(K') > pq$.  Recalling that $pq$-surgery on $K'$ yields $S^3_{p/q}(U) \# S^3_{q/p}(K)$, we note that if $q>0$ then $\maxtb(K') > pq > 0$ and so this is ruled out by Theorem~\ref{thm:main} (unless $S^3_{q/p}(K)=S^3$, in which case $K$ is the unknot and $q=1$ \cite{gordon-luecke-lens}, hence $\maxtb(K')=-1 < pq$ anyway); and if $q<-1$ then this contradicts Proposition~\ref{prop:decomp}, since $S^3_{q/p}(K)$ is not a homology sphere.  Thus $\maxtb(K') \leq pq$ as long as $q\neq -1$.

Given a $tb$-maximizing front diagram for $K$, it is not hard to construct a front for $K'$ by taking $p$ copies of this front, each one shifted off the preceding one by a small distance in the $z$-direction, to produce the $(p,p\cdot\maxtb(K))$-cable of $K$.  If the front for $K$ has writhe $w$ and $c$ cusps, and hence $\maxtb(K) = w - \frac{1}{2}c$, then this $p$-copy has writhe $p^2w  - \frac{p(p-1)}{2}c$ and $pc$ cusps, hence $tb(C_{p,p\cdot \maxtb(K)}(K)) = p^2 \cdot \maxtb(K)$.  If $q < p\cdot \maxtb(K)$ we insert $p \cdot \maxtb(K)-q$ negative $\frac{1}{p}$-twists, each of which has writhe $-(p-1)$ and two cusps and hence adds $-p$ to $tb$, to get $tb(K') = pq$.  If instead $q > p\cdot \maxtb(K)$ then we insert $q-p\cdot\maxtb(K)$ positive $\frac{1}{p}$ twists, each of which has writhe $p-1$ and no cusps, thus adding $p-1$ to $tb$, to get $tb(K') = pq - q + p\cdot \maxtb(K)$.  Thus the front we have constructed provides the desired lower bounds on $\maxtb(K')$ for arbitrary $q$.
\end{proof}

\begin{remark}
The claim that $\maxtb(C_{p,q}(K)) \leq pq$ is actually false for $q=-1$, because if $U$ is the unknot then so is $C_{p,-1}(U)$ for any $p\geq 2$ and so $\maxtb(C_{p,-1}(U)) = -1 > -p$.  One can also see that extending the results of Corollary~\ref{cor:maxtbcables} to $q = -1$ more generally would require removing the possibility of the first conclusion in Theorem~\ref{thm:tb-negative-thm}, since $S^3_{-p}(C_{p,-1}(K)) = S^3_{-p}(U) \# Y$ where $Y=S^3_{-1/p}(K)$.
\end{remark}

It turns out that in the case that $q > p\cdot \maxtb(K)$ in Corollary~\ref{cor:maxtbcables}, we are still sometimes able to determine the maximum Thurston-Bennequin numbers for cables.  We illustrate this for a family of iterated torus knots, namely the ones which are L-space knots, below.

Recall that a knot is an \emph{L-space knot} if it admits a positive L-space surgery, i.e. a rational homology sphere $Y$ with $|H_1(Y;\mathbb{Z})| = \rank \widehat{HF}(Y)$, where $\widehat{HF}$ denotes the hat-flavor of Heegaard Floer homology.  L-space knots are fibered and strongly quasipositive \cite{ni-fibered, hedden-positivity} and thus satisfy $\maxsl(K) = 2g(K)-1$ \cite{etnyre-vhm} and $\maxtb(K) \geq 0$ \cite{rudolph-sliceness}.
We make the following conjecture, which together with Theorem~\ref{thm:main} would immediately imply the main result of \cite{hom-lidman-zufelt}.

\begin{conjecture}\label{conj:maxtblspace}
If $K$ is an L-space knot, then $\maxtb(K) = 2g(K) - 1$.  
\end{conjecture}

In Section~\ref{sec:positive}, we give evidence for this conjecture, including the fact that it holds for Berge knots (Proposition~\ref{prop:berge-tb}), which are the only knots known to have lens space surgeries; and that if it holds for the L-space knot $K$ then it also holds for any cable of $K$ which is also an L-space knot (Proposition~\ref{prop:fibered-2g-1-cables}).  This implies, for example, that $\maxtb(K)=2g(K)-1$ whenever $K$ is an iterated torus knot -- meaning there is a sequence of cables 
\[ K_1 = T_{p_1,q_1}, K_2 = C_{p_2,q_2}(K_1),\dots, K_n = C_{p_n,q_n}(K_{n-1}) \]
with $K = K_n$ -- such that $K_1$ is a positive torus knot and $\frac{q_i}{p_i} \geq 2g(K_{i-1}) - 1$ for all $i\geq 2$.  These conditions on an iterated torus knot are equivalent to it being an L-space knot \cite{hedden-cabling2, hom-cabling}.

\subsection*{Organization}
In Section~\ref{sec:background}, we review the relevant background on contact topology and Stein fillings and prove Proposition~\ref{prop:decomp}.  In Section~\ref{sec:main}, we give a short proof of Theorem~\ref{thm:main}; we then discuss knots which satisfy $\maxtb=2g-1$, show that this holds for positive knots (establishing Theorem~\ref{thm:pos}) and complete the proof of Theorem~\ref{thm:cabling-genus2-positive}.  In Section~\ref{sec:min-d3} we develop some of the background needed to study the case $\maxtb < 0$ and use this to give another proof of Theorem~\ref{thm:main}. Finally, in Section~\ref{sec:negative-tb} we use this background material to prove Theorem~\ref{thm:tb-negative-thm}.

\subsection*{Acknowledgments}
We would like to thank Mohan Bhupal, John Etnyre, Bob Gompf, Cameron Gordon, and Jeremy Van Horn-Morris for helpful discussions.  We would also like to acknowledge that John Etnyre was independently aware some years ago that Theorem~\ref{thm:steinboundary-intro} could be applied to study the cabling conjecture.  Theorem~\ref{thm:cabling-genus2-positive} was completed at the ``Combinatorial Link Homology Theories, Braids, and Contact Geometry'' workshop at ICERM, so we would like to thank the organizers for a productive workshop and the institute for its hospitality.  The first author was supported by NSF RTG grant DMS-0636643.  The second author was supported by NSF postdoctoral fellowship DMS-1204387.

\section{Background}\label{sec:background}

\subsection{Reducible surgeries}
To simplify future references, we collect the list of theorems about reducible surgeries mentioned in the introduction. 
\begin{theorem}
\label{thm:reducible-background}
Let $K$ be a non-trivial knot in $S^3$ and $n,m$ relatively prime integers such that $m \geq 1$.  If $S^3_{n/m}(K)$ is reducible, then   
\begin{enumerate}
\item\label{i:integral} \cite[Theorem~1]{gordon-luecke-integral} $m = 1$;
\item\label{i:lens-summand} \cite[Theorem~3]{gordon-luecke-lens} $S^3_n(K) = L(p,q) \# Y$ for some non-trivial lens space $L(p,q)$, and thus $|n| \geq 2$;
\item\label{i:matignon-sayari} \cite[Theorem~1.1]{matignon-sayari} either $|n| \leq 2g(K) - 1$ or $K$ is not hyperbolic;
\item\label{i:hyperbolic} \cite{moser,scharlemann} if $K$ is not hyperbolic, then it is an $(r,s)$-cable and $n=rs$.
\end{enumerate}
\end{theorem}
Note that since $H_1(S^3_n(K)) = \Z/|n|\Z$, we must have that $|n| = p \cdot |H_1(Y)|$.  Also, in item \eqref{i:hyperbolic} above we consider torus knots to be cable knots, since they are cables of the unknot.

\subsection{Legendrian knots}

For background on Legendrian knots we refer to the survey \cite{etnyre-legendrian} by Etnyre.  In this paper we will only be concerned with Legendrian knots $K$ in the standard tight contact structure $\xi_{\mathrm{std}}$ on $S^3$, i.e.\ knots $K\subset S^3$ which satisfy $TK \subset \xi_{\mathrm{std}}$.  If the front projection of an oriented Legendrian knot $K$ has writhe $w$ and $c_+$ (resp.\ $c_-$) upwardly (resp.\ downwardly) oriented cusps, then its two classical invariants, the Thurston-Bennequin number and rotation number, are defined by
\begin{align*}
tb(K) &= w - \frac{1}{2}(c_+ + c_-), & r(K) &= \frac{1}{2}(c_- - c_+).
\end{align*}
The operations of positive and negative stabilization, which produce a new Legendrian knot $K_\pm$ which is topologically isotopic to $K$ but not Legendrian isotopic to it, change these invariants according to
\begin{align*}
tb(K_\pm) &= tb(K) - 1, & r(K_\pm) = r(K) \pm 1.
\end{align*}
Reversing the orientation of $K$ preserves $tb(K)$ while replacing $r(K)$ with $-r(K)$.

The classical invariants of a Legendrian knot are constrained in general by the Bennequin inequality \cite{bennequin}
\[ tb(K) + |r(K)| \leq 2g(K)-1, \]
where $g(K)$ is the Seifert genus of $K$.  This inequality has been strengthened several times, so that the right side can be replaced by $2g_s(K)-1$, where $g_s(K) \leq g(K)$ is the smooth slice genus \cite{rudolph-quasipositivity}; by $2\tau(K)-1$, where $\tau(K) \leq g_s(K)$ is the Ozsv{\'a}th-Szab{\'o} tau invariant \cite{plamenevskaya-slice-bennequin}; or by $s(K)-1$, where $s(K) \leq 2g_s(K)$ is Rasmussen's $s$ invariant \cite{plamenevskaya-khovanov, shumakovitch}.

\subsection{Stein fillings}
A contact manifold $(Y,\xi)$ is said to be \emph{Stein fillable} if there is a Stein manifold $(X,J)$ with a strictly plurisubharmonic exhausting function $\varphi: X \to \R$ such that $Y = \varphi^{-1}(c)$ for some regular value $c$ of $\varphi$ and $\xi = TY \cap J(TY)$.  The subdomain $(\varphi^{-1}((-\infty,c]), J)$ is a \emph{Stein filling} of $(Y,\xi)$.

Eliashberg \cite{eliashberg-filling} and Gromov \cite{gromov} proved that if $(Y^3,\xi)$ admits a Stein filling, then $\xi$ is tight.  Moreover, Eliashberg characterized the manifolds which admit Stein structures in terms of handlebody decompositions as follows.
\begin{theorem}[Eliashberg, cf.\ \cite{gompf}]\label{thm:steinfill}
Let $X$ be a compact, oriented 4-manifold.  Then $X$ admits a Stein structure if and only if it can be presented as a handlebody consisting of only 0-, 1-, and 2-handles, where the 2-handles are attached along Legendrian knots with framing $tb - 1$ in the unique tight contact structure on $\#^k (S^1 \times S^2)$.
\end{theorem}

In particular, we see that given a knot $K$ in $S^3$, the manifold obtained by attaching a 2-handle to $B^4$ with framing at most $\maxtb(K) - 1$ admits a Stein structure, since by stablizing, we can obtain a Legendrian representative with $tb(K) = n$ for any $n \leq \maxtb(K)$. 

Since lens spaces have metrics of positive scalar curvature, the topology of their Stein fillings is heavily constrained.  
\begin{theorem}[Lisca \cite{lisca-fillings}]\label{thm:lisca-fillings}
Let $(X,J)$ be a Stein filling of a lens space.  Then $b^+_2(X) = 0$.  
\end{theorem}

We also recall the definition of the $d_3$ invariant, due to Gompf \cite{gompf}, of oriented plane fields $\xi$ with torsion Chern class on a closed, oriented $3$-manifold.
\begin{theorem}\label{thm:d3-invariant}
Let $(X,J)$ be an almost complex manifold with $\partial X = Y$ and $\xi = TY \cap J(TY)$.  If $c_1(\xi)$ is torsion, then
\begin{equation}\label{eqn:d3-invariant}
d_3(\xi) = \frac{c_1(X,J)^2 - 3\sigma(X) - 2\chi(X)}{4}
\end{equation}
is an invariant of the homotopy class of $\xi$ as an oriented plane field.
\end{theorem}

All of the three-manifolds we will be concerned with in this paper will be rational homology spheres, so for any oriented plane field $\xi$ that we will consider, $c_1(\xi)$ will be torsion.   
\begin{example}
\label{ex:d3-contractible}
If $(Y,\xi)$ is the boundary of a contractible Stein manifold $X$, then $d_3(\xi) = -\frac{1}{2}$.  Examples include the tight contact structure on $S^3$, which is filled by $B^4$.  
\end{example}

Now, if $(Y_i,\xi_i)$ bounds an almost complex manifold $(X_i,J_i)$ for $i=1,2$, then we can glue a Weinstein $1$-handle to $X_1 \sqcup X_2$ to exhibit the boundary sum $X_1 \natural X_2$ as an almost complex manifold with boundary $(Y_1 \# Y_2, \xi_1 \# \xi_2)$, and so
\begin{equation}\label{eqn:d3-sum}
d_3(\xi_1 \# \xi_2) = d_3(\xi_1) + d_3(\xi_2) + \frac{1}{2}. 
\end{equation}
Combining this fact with Example \ref{ex:d3-contractible}, we see that if $(Y_2,\xi_2)$ is the boundary of a contractible Stein manifold then 
\begin{equation}\label{eqn:d3-contractible}
d_3(\xi_1\#\xi_2)=d_3(\xi_1).
\end{equation}

\subsection{Reducible contact manifolds}

As mentioned in the introduction, our main input will be the following characterization of Stein fillings of non-prime contact three-manifolds, which will enable us to prove Proposition~\ref{prop:decomp}.

\begin{theorem}[Eliashberg \cite{eliashberg-filling, cieliebak-eliashberg}]\label{thm:steinboundary}
Suppose that $(X,J)$ is a Stein filling of a non-prime contact $3$-manifold $(Y_1 \# Y_2, \xi_1 \# \xi_2)$.  Then $(X,J)$ decomposes as a boundary sum $(X_1,J_1) \natural (X_2,J_2)$, where $(X_i,J_i)$ is a Stein filling of $(Y_i,\xi_i)$.\end{theorem}

\begin{proof}[Proof of Proposition~\ref{prop:decomp}]
By Theorem~\ref{thm:reducible-background}, a reducible surgery on a non-trivial knot is necessarily integral and has a non-trivial lens space summand.  Let $X$ be the 2-handlebody obtained by attaching an $n$-framed 2-handle to the four-ball along $K$.  Observe that $X$ is simply-connected and has intersection form $\langle n \rangle$.  By Theorem~\ref{thm:steinfill}, if $n \leq \maxtb(K) - 1$, then $X$ admits a Stein structure $J$.  Now, Theorem~\ref{thm:steinboundary} implies that if $(X,J)$ is a Stein filling of $(S^3_n(K),\xi) = (L(p,q),\xi_1) \# (Y,\xi_2)$, for $Y \neq S^3$, then $X$ decomposes as a boundary sum, say $X = (W_1,J_1) \natural (W_2,J_2)$, where $(W_1,J_1)$ is a Stein filling of $(L(p,q),\xi_1)$ and $(W_2,J_2)$ is a Stein filling of $(Y,\xi_2)$.  

It is clear that $W_1$ and $W_2$ are simply-connected.  Since $\pi_1(W_1) = 0$ and $H_1(\partial W_1) \neq 0$, we must have $H_2(W_1) \neq 0$.  Then $H_2(W_1)$ is a summand of $H_2(W_1) \oplus H_2(W_2) \cong H_2(X) \cong \Z$, so $H_2(W_1)$ carries $H_2(X)$.  Thus $W_1$ has intersection form $\langle n \rangle$ and $H_2(W_2) = 0$.  Consequently, we have $H_1(\partial W_1) = \Z/|n|\Z$, and so $|n| = |p|$.  Since $\pi_1(W_2) = H_2(W_2) = 0$, and $W_2$ has no 3- or 4-handles by Theorem~\ref{thm:steinfill}, we see that $W_2$ is contractible and thus $H_1(Y) = H_1(\partial W_2) = 0$.  

In summary, we have $S^3_n(K) = L(|n|,q) \# Y$, where $Y$ is an integer homology sphere, $(W_1,J_1)$ provides a simply-connected Stein filling of $L(|n|,q)$ with intersection form $\langle n \rangle$, and $(W_2,J_2)$ provides a contractible Stein filling of $Y$.  If $S^3_n(K)$ has at least three nontrivial connected summands then all but one of them are lens spaces by Theorem~\ref{thm:reducible-background}, and since $Y$ is a homology sphere we conclude that it must be irreducible.  Finally, if $n > 0$ then $b^+_2(W_1) > 0$, contradicting Theorem~\ref{thm:lisca-fillings}, so it follows that $n < 0$.
\end{proof}

\subsection{Tight contact structures on lens spaces}\label{subsec:tight-lens}

We recall the classification of tight contact structures on the lens space $L(p,q)$, due to Giroux and Honda.  We use the convention here and from now on that $L(p,q)$ denotes $-\frac{p}{q}$-surgery on the unknot, and that $1 \leq q < p$.

\begin{theorem}[\cite{giroux-lens,honda}]
\label{thm:lens-space-classification}
If $-\frac{p}{q}$ has continued fraction
\[ [a_1,a_2,\dots,a_n] := a_1 - \frac{1}{a_2 - \frac{1}{\dots - \frac{1}{a_n}}}, \]
where $a_i \leq -2$ for all $i$, then Legendrian surgery on a chain of topological unknots of length $n$ in which the $i$th unknot has Thurston-Bennequin number $a_i+1$ and rotation number
\[ r_i \in \{a_i+2, a_i+4, a_i+6,\dots, |a_i|-2\} \]
produces a tight contact structure on $L(p,q)$.  This construction gives a bijection between the set of such tuples $(r_1,\dots,r_n)$, which has $\prod_i (|a_i|-1)$ elements, and the set of tight contact structures on $L(p,q)$ up to isotopy.
\end{theorem}

The Legendrian surgery construction of Theorem~\ref{thm:lens-space-classification} also produces a Stein filling $(X,J_{\rv})$ of each $(L(p,q),\xi_{\rv})$, where $\xi_{\rv}$ is the contact structure determined by the ordered set of rotation numbers
\[ \rv = \langle r_1,r_2,\dots,r_n \rangle \]
once we orient each unknot in the chain so that every pair of adjacent unknots has linking number $1$; we make this choice of orientation to simplify the linking matrix, and consequently the matrix presentation of the intersection form for $X$.  Then $\sigma(X) = -n$ since $X$ is necessarily negative definite, and $\chi(X) = n+1$, so
\begin{equation}\label{eq:d_3-lens-gompf}
d_3(\xi_{\rv}) = \frac{c_1(X,J_{\rv})^2 + n - 2}{4}. 
\end{equation}
According to Gompf \cite{gompf}, the Chern class in this formula is Poincar\'{e} dual to $\sum_{i=1}^n r_i [D_i] \in H_2(X,\partial X)$, where each disk $D_i$ is the cocore of the $2$-handle attached to $\partial B^4 = S^3$ along the $i$th unknot.  We will use this description later to compute $c_1(X,J_{\rv})^2$.

\begin{remark}
\label{rem:xi-conjugate}
The tight contact structures on $L(p,q)$ come in conjugate pairs $\xi = \xi_{\rv}$ and $\bar{\xi} = \xi_{-\rv}$, which are isomorphic as plane fields but with opposite orientations.  The discussion above implies that the corresponding almost complex structures satisfy $c_1(X,J_{\rv}) = -c_1(X,J_{-\rv})$, and hence that $d_3(\xi) = d_3(\bar{\xi})$.  Conjugation acts as an involution on the set of tight contact structures on $L(p,q)$, with at most one fixed point ($\rv=\langle 0,0,\dots,0 \rangle$), which satisfies $d_3(\xi)=\frac{n-2}{4}$ and which only exists if all of the $a_i$ are even.
\end{remark}

\begin{remark}
\label{rem:xi-can}
There is a canonical contact structure $\xican$ on $L(p,q)$, defined as follows: the standard contact structure $\xi_\mathrm{std}$ on $S^3$ is $\ZZ/p\ZZ$--equivariant under the action used to define $L(p,q)$, and $\xican$ is defined as the quotient of $\xi_\mathrm{std}$ under this action.  We know that $\xican$ is the contact structure $\xi_{\langle |a_1|-2,|a_2|-2,\dots,|a_n|-2\rangle}$, in which each $r_i$ is as large as possible \cite[Proposition 3.2]{ozbagci-horizontal} (see also \cite[Section 7]{bhupal-ozbagci}).
\end{remark}

We can use Theorem \ref{thm:lens-space-classification} to bound the number of tight contact structures on $L(p,q)$ as follows.

\begin{proposition}
\label{prop:count-xi}
Take relatively prime integers $p$ and $q$, $p > q \geq 1$, and write $-\frac{p}{q} = [a_1,\dots,a_n]$ with each $a_i \leq -2$.  If $m = \min_i |a_i|$, then $L(p,q)$ has at most 
\[ \frac{m-1}{m}(p-(n-1)(m-1)^{n-1}) \]
tight contact structures up to isotopy, with equality if and only if either $n \leq 2$ or $p=q+1$.
\end{proposition}

\begin{proof}
We remark that if $n \geq 2$ and $-\frac{s}{r} = [a_2,\dots,a_n]$, then 
\begin{equation}\label{eq:numerator-recursion}
-\frac{p}{q} = -\frac{|a_1|s-r}{s},
\end{equation} hence $q=s$ (so in fact $-\frac{q}{r} = [a_2,\dots,a_n]$) and $p=|a_1|q-r$.  We then note that
\[ p - q = (|a_1|-1)q - r \geq (|a_1|-1)(q-r) \geq (m-1)(q-r), \]
with equality only if $|a_1|=m=2$.  Applying this repeatedly gives $p - q \geq (m-1)^n$ with equality if and only if either $|a_i|=m=2$ for all $i$, in which case $p=q+1$, or $n=1$. 

We now prove the proposition by induction on $n$: certainly when $n=1$ we must have $-\frac{p}{q} = [-p] = -\frac{p}{1}$, so $m=p$, and $L(p,1)$ has exactly $p-1 = \frac{m-1}{m}\cdot p$ contact structures by Theorem~\ref{thm:lens-space-classification}. Suppose that $n \geq 2$ and that $a_i=-m$ for some $i>1$.   Then
\begin{align*}
|\operatorname{Tight}(L(p,q))| &= (|a_1|-1) \cdot |\operatorname{Tight}(L(q,r))| \\
&\leq \frac{(|a_1|-1)\cdot (m-1)(q-(n-2)(m-1)^{n-2})}{m} \\
&= \frac{m-1}{m}\left((|a_1|q-q) - (|a_1|-1)(n-2)(m-1)^{n-2}\right) \\
&\leq \frac{m-1}{m}\left(p - (q-r) - (n-2)(m-1)^{n-1}\right) \\
&\leq \frac{m-1}{m}\left(p-(n-1)(m-1)^{n-1}\right),
\end{align*}
where we use the facts that $p=|a_1|q-r$, $|a_1| \geq m$, $n\geq 2$, and $q-r \geq (m-1)^{n-1}$ as shown above.  If we have equality at each step then $q-r = (m-1)^{n-1}$, hence either $n=2$ or $q=r+1$; in the latter case we have $m=2$, and assuming $n>2$ we must have $|a_1|-1=m-1$ as well, so $a_i=-2$ for all $i$, and thus $p=q+1$.  Conversely, if $n=2$ then it is easy to see that equality is preserved, and likewise if $p=q+1$ since this implies $a_i = -2$ for all $i$.

If on the other hand $n \geq 2$ but only $a_1$ is equal to $-m$, then we apply the same argument to $L(p,q')$ where $-\frac{p}{q'} = [a_n,\dots,a_1]$ and observe that this is homeomorphic to $L(p,q)$, since they are presented by surgery on the same chain of unknots viewed from two different perspectives.  This completes the induction.
\end{proof}


\section{Reducible Legendrian surgeries for $\maxtb(K) \geq 0$}\label{sec:main}

\subsection{A proof of Theorem~\ref{thm:main}}
\label{ssec:proof-main}

Proposition~\ref{prop:decomp} guarantees that associated to a reducible Legendrian surgery is a certain Stein filling of a lens space, and consequently a tight contact structure on this lens space.  We will prove Theorem~\ref{thm:main} by showing that reducible Legendrian surgeries on knots with $\maxtb \geq 0$ produce too many tight contact structures in this fashion, appealing to Giroux and Honda's classification of tight contact structures on $L(p,q)$ (Theorem~\ref{thm:lens-space-classification}).  We recall our conventions that $L(p,q)$ is $-\frac{p}{q}$-surgery on the unknot and that $1 \leq q < p$.

\begin{proposition}
\label{prop:reducible-d3-value}
Let $K$ be a knot, and suppose that $S^3_n(K)$ is reducible for some $n \leq \maxtb(K)-1$; write $n=-p$ for some $p \geq 2$.  If $S^3_n(K) = L(p,q) \# Y$, then Legendrian surgery on any representative of $K$ with $tb = 1-p$ and rotation number $r$ induces a tight contact structure $\xi$ on $L(p,q)$ with $d_3(\xi) = -\frac{1}{4p}(r^2+p)$.
\end{proposition}

\begin{proof}
By Theorem~\ref{thm:steinfill}, the reducible Legendrian surgery gives us a Stein filling $(X,J)$ of a reducible contact manifold
\[ (S^3_n(K),\xi_K) = (L(p,q), \xi) \# (Y,\xi'), \]
where $(Y,\xi')$ bounds a contractible Stein manifold by Proposition~\ref{prop:decomp}.  Then equation \eqref{eqn:d3-contractible} says that $d_3(\xi) = d_3(\xi_K)$, so it remains to compute $d_3(\xi_K)$.  By Proposition~\ref{prop:decomp}, $\sigma(X)=-1$ (since $n<0$) and $\chi(X)=2$.  Thus $d_3(\xi_K) = \frac{1}{4}(c_1(X,J)^2 - 1)$.

In order to compute $c_1(X,J)^2$, we observe that $c = PD(c_1(X,J))$ is the class $r[D] \in H_2(X,\partial X)$ \cite{gompf}, where $D$ is the cocore of the $2$-handle attached along $K$.  Then $H_2(X)=\ZZ$ is generated by a surface $\Sigma$ of self-intersection $n$, obtained by capping off a Seifert surface for $K$ with the core of the $2$-handle, and the map $H_2(X)\to H_2(X,\partial X)$ sends $[\Sigma] \mapsto n[D]$.  In particular, it sends $-r[\Sigma]$ to $-rn[D] = pc$, and so
\[ p^2c^2 = (-r[\Sigma])^2 = r^2n = -r^2p, \]
or $c^2 = -\frac{r^2}{p}$.   We conclude that $d_3(\xi) = d_3(\xi_K) = \frac{1}{4}\left(-\frac{r^2}{p}-1\right)$, as desired.
\end{proof}

At this point we can give a simple proof of Theorem~\ref{thm:main}, which says that if $\maxtb(K)\geq 0$ then $n$-surgery on $K$ is irreducible for all $n < \maxtb(K)$.

\begin{proof}[Proof of Theorem~\ref{thm:main}]
Suppose that $S^3_n(K)$ is reducible for some $n < \maxtb(K)$.  We know by Proposition \ref{prop:decomp} that $n\leq -2$ and the reducible manifold has a summand of the form $L(p,q)$ with $p = -n$.  Since $K$ has a Legendrian representative with $tb = \maxtb(K) \geq 0$, and $tb+r$ is odd, after possibly reversing the orientation of $K$, it has a representative with $tb=0$ and $r = r_0 \geq 1$.  We can stabilize this representative $p-1$ times with different choices of signs to get representatives with $tb=1-p$ and
\[ r \in \{r_0 - p + 1, r_0 - p + 3, r_0 - p + 5, \dots, r_0 + p - 1\}, \]
and by reversing orientation we also get one with $tb=1-p$ and $r=-r_0-p+1$.  Thus the Legendrian representatives of $K$ with $tb=1-p$ collectively admit at least $p+1$ different rotation numbers, hence at least $\left\lceil \frac{p+1}{2} \right\rceil$ values of $r^2$.

For each value of $r$ as above, Proposition \ref{prop:reducible-d3-value} says that $L(p,q)$ admits a tight contact structure $\xi$ with $d_3(\xi) = -\frac{r^2+p}{4p}$.  This value of $d_3(\xi)$ is uniquely determined by $r^2$, so the set of rational numbers
\begin{equation}
\label{eq:tight-d3-set}
\left\{ d_3(\xi) \mid \xi \in \operatorname{Tight}(L(p,q)) \right\}
\end{equation}
has at least $\left\lceil \frac{p+1}{2} \right\rceil$ elements.

Now we know from Proposition~\ref{prop:count-xi} that $L(p,q)$ has at most $p-1$ tight contact structures.  Moreover, by Remark~\ref{rem:xi-conjugate} all but at most one of them come in conjugate pairs.  Observe that conjugate contact structures have the same $d_3$ invariant, and so the set \eqref{eq:tight-d3-set} has at most $\left\lceil \frac{p-1}{2} \right\rceil$ elements.  We conclude that
\[ \left\lceil \frac{p+1}{2} \right\rceil \leq \left\lceil \frac{p-1}{2} \right\rceil, \]
which is absurd.
\end{proof}


\subsection{Knots with $\maxtb = 2g-1$}
\label{sec:positive}

In this subsection we discuss the question of which nontrivial knots can have $\maxtb(K) = 2g(K)-1$, where $g(K)$ is the Seifert genus.  We have already shown that reducible surgeries on such knots must have slope at least $2g(K)-1$.  Recall from Theorem~\ref{thm:reducible-background} that if $K$ is also hyperbolic, then Matignon-Sayari showed that reducible surgeries on $K$ have slope at most $2g(K)-1$, so then $n$-surgery on $K$ cannot be reducible unless $n=2g(K)-1$.

\begin{proposition}
\label{prop:max-tb-positive-knot}
If $K$ is a positive knot, then $\maxtb(K) = 2g(K)-1$.
\end{proposition}

\begin{proof}
Hayden--Sabloff \cite{hayden-sabloff} proved that if $K$ is positive then it admits a Lagrangian filling, hence by a theorem of Chantraine \cite{chantraine-concordance} it satisfies $\maxtb(K) = 2g_s(K)-1$, where $g_s(K)$ is the slice genus of $K$, and Rasmussen \cite{rasmussen-s} proved that $g_s(K)=g(K)$ for positive knots.
\end{proof}

It follows from this and item~\eqref{i:hyperbolic} of Theorem~\ref{thm:reducible-background} that if $n$-surgery on a positive knot $K$ is reducible, then either $K$ is a cable and $n$ is the cabling slope, or $K$ is hyperbolic and $n=2g(K)-1$; thus we have proved Theorem~\ref{thm:pos}.  (If $K$ is also hyperbolic, the claim that there are no essential punctured projective planes in its complement follows exactly as in \cite[Corollary 1.5]{hom-lidman-zufelt}.)  In the case $g(K)=2$, we can use Heegaard Floer homology to eliminate this last possibility as well.

\begin{theorem}
Positive knots of genus at most 2 satisfy the cabling conjecture.
\end{theorem}

\begin{proof}
The cabling conjecture is true for genus 1 knots by \cite{boyer-zhang} (see also \cite{matignon-sayari, hom-lidman-zufelt}).  If $n$-surgery on the genus 2 positive knot $K$ is a counterexample then $K$ must be hyperbolic by Theorem~\ref{thm:reducible-background} (in particular, $K$ is prime) and $n=2g(K)-1 = 3$ by Theorem~\ref{thm:pos}.  As a positive knot of genus 2, $K$ is quasi-alternating \cite{jong-kishimoto}, hence it has thin knot Floer homology \cite{manolescu-ozsvath}.  The signature of $K$ is at most $-4$, since positive knots satisfy $\sigma(K) \leq -4$ unless they are pretzel knots \cite[Corollary 1.3]{przytycki-taniyama} and the cabling conjecture is known for pretzel knots \cite{luft-zhang} (in fact, for all Montesinos knots).  Since $\frac{|\sigma(K)|}{2}$ is a lower bound for the slice genus of $K$, and hence for $g(K)=2$, we have $\sigma(K) = -4$.

We claim that $K$ cannot be fibered.  Indeed, Cromwell \cite[Corollary 5.1]{cromwell} showed that fibered homogeneous knots have crossing number at most $4g(K)$, and since positive knots are homogeneous we need only check the knots with at most 8 crossings in KnotInfo \cite{cha-livingston} to verify that the $(2,5)$-torus knot is the only prime, fibered, positive knot of genus $2$, and it is not hyperbolic.  Since L-space knots are fibered \cite{ni-fibered}, the reducible $3$-surgery on $K$ cannot be an L-space.  Its lens space summand has order dividing 3, so it must be $L(3,q)$ for some $q$, and if we write $S^3_3(K) = L(3,q) \# Y$ for some homology sphere $Y$, then it follows from the K\"unneth formula for $\hfhat$ \cite{ozsvath-szabo-properties} that $Y$ is not an L-space.

Since $K$ is $HFK$-thin, the computation of $\hfp(S^3_3(K))$, the plus-flavor of Heegaard Floer homology, was carried out in the proof of \cite[Theorem 1.4]{ozsvath-szabo-alternating}; for the claim that the surgery coefficient $n=2g(K)-1$ is ``sufficiently large,'' see \cite[Section 4]{ozsvath-szabo-knots}, in particular Corollary~4.2 and Remark~4.3.  The result (up to a grading shift in each $\spc$ structure) depends only on the signature $\sigma(K)=-4$ and some integers $b_i$ determined by the Alexander polynomial of $K$ as follows:
\begin{align*}
\hfp(S^3_3(K), 0) &\cong \mathcal{T}^+_{-5/2} \oplus \ZZ^{b_0}_{(-7/2)} \\
\hfp(S^3_3(K), 1) &\cong \mathcal{T}^+_{-2} \oplus \ZZ^{b_1}_{(-2)}
\end{align*}
and $\hfp(S^3_3(K),2) \cong \hfp(S^3(K),1)$, where the numbers $0,1,2$ denote the different $\spc$ structures on $S^3_3(K)$ and the subscripts on the right denote the grading of either the lowest element of the tower $\mathcal{T}^+ = \ZZ[U,U^{-1}]/U \cdot \ZZ[U]$ or the $\ZZ^{b_i}$ summand.  The K\"{u}nneth formula for Heegaard Floer homology implies that each $\hfp(S^3_3(K),i)$ should be isomorphic to $\hfp(Y)$ as a relatively graded $\ZZ[U]$-module (with an absolute shift determined by the correction terms of $L(3,q)$), and in particular we must have $b_i \neq 0$ since $Y$ is not an L-space.  Thus the $\ZZ^{b_i}$ summands are nontrivial.  However, we see that $HF^+(S^3_3(K),0)$ and $HF^+(S^3_3(K),1)$ are not isomorphic as relatively-graded groups, by comparing the gradings of the $\ZZ^{b_i}$ summand to the gradings of the tower.  Thus, the corresponding $\hfp(S^3_3(K),i)$ cannot both be isomorphic to $\hfp(Y)$ and we conclude that $S^3_3(K)$ is not reducible after all.
\end{proof}

In general the condition $\maxtb = 2g-1$ can be fairly restrictive, as shown by the following.

\begin{proposition}
If $K$ is fibered and $\maxtb(K)=2g(K)-1$, then $K$ is strongly quasipositive.
\end{proposition}

\begin{proof}
Livingston \cite{livingston} and Hedden \cite{hedden-positivity} showed that for fibered knots we have $\tau(K) = g(K)$ if and only if $K$ is strongly quasipositive, where $\tau$ is the Ozsv{\'a}th--Szab{\'o} concordance invariant, which always satisfies $|\tau(K)| \leq g(K)$ \cite{ozsvath-szabo-four-ball}.  On the other hand, Plamenevskaya \cite{plamenevskaya-slice-bennequin} proved that $tb(K)+|r(K)| \leq 2\tau(K)-1$ for any Legendrian representative of $K$.  In particular, if $K$ is fibered and not strongly quasipositive then $\maxtb(K) \leq 2\tau(K)-1 < 2g(K)-1$.
\end{proof}

It is not true that all fibered, strongly quasipositive knots satisfy $\maxtb(K)=2g(K)-1$.  One example is the $(3,2)$-cable of the right-handed trefoil $T$: letting $K = C_{3,2}(T)$, Etnyre and Honda \cite{etnyre-honda-cabling} showed that $\maxtb(K) = 6$ (which also follows from Corollary~\ref{cor:maxtbcables}) but $\maxsl(K) = 7 = 2g(K)-1$, and since $K$ is fibered the latter implies by \cite{hedden-positivity} that it is strongly quasipositive.  Etnyre, LaFountain, and Tosun \cite{etnyre-lafountain-tosun} provided many other examples as cables $C_{r,s}(T_{p,q})$ of positive torus knots, but in all such cases we have $\maxtb(K) < 2g(K)-1$ only if $\frac{s}{r} < 2g(T_{p,q})-1$, in which case $K$ is not an L-space knot \cite{hom-cabling}.

For a hyperbolic example, let $K$ be the closure of the strongly quasipositive 3-braid
\begin{align}
\label{eq:hyperbolic-braid}
\beta = \sigma_1^2 \sigma_2^3 (\sigma_1 \cdot \sigma_1 \sigma_2 \sigma_1^{-1} \cdot \sigma_2)^3,
\end{align}
cf.\ \cite[Remark~5.1]{stoimenow}.  Since $K$ is the closure of a 3-braid and its Alexander polynomial
\[ \Delta_K(t) = t^{-6} - 2t^{-4} + 3t^{-3} - 2t^{-2} + 1 - 2t^2 + 3t^3 - 2t^4 + t^6  \]
is monic, we know that $K$ is fibered \cite[Corollary~4.4]{stoimenow} with Seifert genus $g(K)= 6$.
\begin{lemma}
The knot $K$ defined as the closure of \eqref{eq:hyperbolic-braid} is hyperbolic.
\end{lemma}
\begin{proof}
It suffices to check that $K$ is not a satellite, since $\Delta_K(t)$ is not the Alexander polynomial of a torus knot.  If $K$ is a satellite with companion $C$ and pattern $P$, and $P$ has winding number $w$ in the solid torus, then $\Delta_K(t) = \Delta_C(t^w)\Delta_P(t)$.  Since $K$ is fibered, both $C$ and $P$ are fibered and $w\neq 0$ \cite[Theorem~1]{hirasawa-murasugi-silver}, and in particular $\Delta_C(t^w)$ is not constant since it has degree $w\cdot g(C)$.  Since $\Delta_K(t)$ is irreducible, it follows that $\Delta_P(t)=1$ and $\Delta_K(t) = \Delta_C(t^w)$, which by inspection implies $w=\pm 1$.  Since $P$ is fibered with trivial Alexander polynomial, it is unknotted in $S^3$; but since it also has winding number 1 it must then be isotopic to the core of the solid torus \cite[Corollary~1]{hirasawa-murasugi-silver}, and so $K$ cannot be a nontrivial satellite.
\end{proof}
We have the bound $\maxtb(K) \leq -\operatorname{max-deg}_a F_K(a,z)-1 = 10$ where $F_K$ is the Kauffman polynomial of $K$ \cite{rudolph-congruence}, and so $\maxtb(K) < 2g(K)-1$.  We note that $K$ is not an L-space knot since the coefficients of $\Delta_K(t)$ are not all $\pm 1$ \cite{ozsvath-szabo-lens}, so Conjecture~\ref{conj:maxtblspace} remains intact.

As evidence for Conjecture~\ref{conj:maxtblspace}, we show that it is satisfied by all knots which are known to admit positive lens space surgeries, i.e.\ the twelve families of Berge knots \cite{berge-conjecture}.  (Berge knots are strongly invertible by a result of Osborne \cite{osborne}, cf.\ \cite{watson}, so they are already known to satisfy the cabling conjecture \cite{eudave-munoz}.)

\begin{proposition}
\label{prop:berge-tb}
If $K$ is a Berge knot, reflected if necessary so that it has a positive lens space surgery, then $\maxtb(K) = 2g(K)-1$.
\end{proposition}

\begin{proof}
Families I--VI are the Berge-Gabai knots \cite{berge-solid-torus, gabai-solid-torus}, which are knots in $S^1 \times D^2$ with nontrivial $S^1 \times D^2$ surgeries, and these are known (after possibly reflecting as mentioned above) to be braid positive since they are either torus knots or 1-bridge braids.  Families VII and VIII are knots on the fiber surface of a trefoil or figure eight, respectively, and Baker \cite[Appendix B]{baker-thesis} showed that they are braid positive as well.  Thus in each of these cases the result follows from Proposition~\ref{prop:max-tb-positive-knot}.  The remaining ``sporadic'' knots, families IX--XII, have $\maxtb=2g-1$ because they are all divide knots \cite{yamada-sporadic}, hence they satisfy $\maxtb(K) = 2g_s(K)-1$ \cite{ishikawa}, and as L-space knots they have $g_s(K)=g(K)$ \cite{ozsvath-szabo-lens}, and the result follows.
\end{proof}

We compile further evidence for Conjecture~\ref{conj:maxtblspace} by reducing it to the case of non-cabled knots as follows.  We recall that L-space knots are fibered \cite{ni-fibered}.  Moreover, torus knots are L-space knots if and only if they are positive, in which case they satisfy $\maxtb(K)=2g(K)-1$ by Proposition~\ref{prop:max-tb-positive-knot}.  Finally, the cable $C_{p,q}(K)$ of some nontrivial $K$ is an L-space knot if and only if $K$ is an L-space knot and $\frac{q}{p} \geq 2g(K)-1$ \cite{hedden-cabling2, hom-cabling}, in which case we can apply the following.

\begin{proposition}
\label{prop:fibered-2g-1-cables}
If $K$ is fibered and nontrivial and $\maxtb(K)=2g(K)-1$, then $\maxtb(C_{p,q}(K)) = 2g(C_{p,q}(K)) - 1$ whenever $\frac{q}{p} \geq 2g(K)-1$.
\end{proposition}

\begin{proof}
Because $C_{p,q}(K) = C_{-p,-q}(K)$, we assume without loss of generality that $p, q > 0$; then the inequality $\frac{q}{p} \geq 2g(K)-1 = \maxtb(K)$ implies $q > p\cdot\maxtb(K)$, since $\gcd(p,q)=1$.  Corollary~\ref{cor:maxtbcables} then says that $\maxtb(K') \geq pq - (q-p\cdot \maxtb(K))$, where $K' = C_{p,q}(K)$.

Since $K$ is fibered, its cable $K'$ is as well and their Seifert genera are known to be related by $g(K') = p\cdot g(K) + \frac{(p-1)(q-1)}{2}$, which is equivalent to
\[ 2g(K')-1 = pq - (q - p(2g(K)-1)) = pq - (q-p\cdot\maxtb(K)) \leq \maxtb(K'). \]
But Bennequin's inequality implies that $\maxtb(K') \leq 2g(K')-1$, so the two must be equal.
\end{proof}

\section{Minimal $d_3$ invariants of tight contact structures}
\label{sec:min-d3}

Although the proof of Theorem~\ref{thm:main} in Section~\ref{ssec:proof-main} is short and simple, it does not seem easily adaptable to the case $\maxtb(K) < 0$.  In this section we will undertake a more careful study of the values of $d_3$ invariants of tight contact structures on $L(p,q)$, which will be used in Section~\ref{sec:negative-tb} to prove Theorem~\ref{thm:tb-negative-thm}.  The main results of this section are Proposition~\ref{prop:d3-xican-minimal}, which asserts that out of all tight contact structures on $L(p,q)$, the ones which minimize $d_3(\xi)$ are precisely $\xican$ and its conjugate; and Proposition~\ref{prop:f-recurrence-eqn}, which presents a recurrence relation for $d_3(\xican)$.  As a quick application, we conclude this section with a second proof of Theorem~\ref{thm:main}.

\subsection{Tight contact structures on $L(p,q)$ with minimal $d_3$ invariants}
\label{ssec:minimal-d3}

Following Section~\ref{subsec:tight-lens}, each tight contact structure $\xi_{\rv}$ on $L(p,q)$ is equipped with a Stein filling $(X,J_{\rv})$.  We will determine $d_3(\xi_{\rv})$ by using \eqref{eq:d_3-lens-gompf}, which means that we must compute $c_1(X,J_{\rv})^2$.  This requires a slight generalization of the argument of Proposition~\ref{prop:reducible-d3-value} as follows (see e.g. \"{O}zba\u{g}c\i--Stipsicz \cite{ozbagci-stipsicz}).

We note that $H_2(X) = \ZZ^n$ is generated by classes $[\Sigma_1],\dots,[\Sigma_n]$, where each $\Sigma_i$ is generated by taking a Seifert surface for the $i$th unknot in the chain and capping it off with the core of the corresponding $2$-handle.  In this basis the intersection form on $X$ is given by the linking matrix
\[ M_{p/q} = \left(\begin{array}{cccccc}
a_1 & 1 & 0 & \cdots & 0 & 0 \\
1 & a_2 & 1 & \cdots & 0 & 0 \\
0 & 1 & a_3 & \cdots & 0 & 0 \\
\vdots & \vdots & \vdots & \ddots & \vdots & \vdots \\
0 & 0 & 0 & \cdots & a_{n-1} & 1 \\
0 & 0 & 0 & \cdots & 1 & a_n
\end{array}\right). \]
Since $M_{p/q}$ is negative definite and also presents $H_1(\partial X) = H_1(L(p,q)) = \ZZ/p\ZZ$, we know that $\det(M_{p/q}) = (-1)^n p$.  In general, if $M$ is a $k\times k$ tridiagonal matrix of this form with entries $b_1,\dots,b_k$ along the diagonal, we will write
\[ d(b_1,\dots,b_k) = |\det(M)|. \]
For example, $d(a_1,\dots,a_n) = p$.

The group $H_2(X,\partial X)$ is generated by homology classes $[D_1],\dots,[D_n]$ of disks, where each $D_i$ is the cocore of the $i$th handle, and according to Gompf \cite{gompf} we have
\[ c_1(X,J_{\rv}) = PD\left(\sum_{i=1}^n r_i [D_i]\right). \]
If $c = PD(c_1(X,J_{\rv})) \in H_2(X,\partial X)$, then $pc$ is represented by a surface whose boundary is nullhomologous in $H_1(\partial X) = \ZZ/p\ZZ$, so $pc$ lifts to a class $C \in H_2(X)$.  If $C$ is represented by a vector $\sv$ in the given basis of $H_2(X)$ then we have $M_{p/q}\sv = p\rv$, and so
\[ p^2 c^2 = C^2 = \sv^T M_{p/q} \sv = (p\cdot M_{p/q}^{-1}\rv)^T M_{p/q} (p\cdot M_{p/q}^{-1} \rv), \]
hence $c^2 = \rv^T \cdot \left(M^{-1}_{p/q}\right) \cdot \rv$.

Let $A_{p/q} = -p \cdot M_{p/q}^{-1}$.  Using equation \eqref{eq:d_3-lens-gompf} and the above computation of $c_1(X,J_{\rv})^2$, we can determine the $d_3$ invariant of the contact structure $\xi_{\rv}$ on $L(p,q)$ by the formula
\begin{equation}
\label{eq:d3-formula-apq}
d_3(\xi_{\rv}) = \frac{-\frac{1}{p} (\rv^T A_{p/q} \rv) + n-2}{4}.
\end{equation}
We note that $A_{p/q}$ is positive definite, since $M_{p/q}^{-1}$ is negative definite.  Moreover, the entries of $A_{p/q}$ are all integers because $\det(M_{p/q}) = \pm p$, and so these entries are simply the cofactors of the integer matrix $M_{p/q}$ up to sign.  In fact, it is not so hard to explicitly determine the entries of $A_{p/q}$, which when combined with \eqref{eq:d3-formula-apq} will be key for our computation of $d_3$ invariants.

\begin{proposition}
\label{prop:a-pq-entries}
If $-\frac{p}{q} = [a_1,\dots,a_n]$, then $A_{p/q}$ has the form
\[
\left(\begin{array}{ccccc}
d()d(a_2,\dots,a_n) & d()d(a_3,\dots,a_n) & d()d(a_4,\dots,a_n) & \dots & d()d() \\
d()d(a_3,\dots,a_n) & d(a_1)d(a_3,\dots,a_n) & d(a_1)d(a_4,\dots,a_n) & \dots & d(a_1)d() \\
d()d(a_4,\dots,a_n) & d(a_1)d(a_4,\dots,a_n) & d(a_1,a_2)d(a_4,\dots,a_n) & \dots & d(a_1,a_2)d() \\
\vdots & \vdots & \vdots & \ddots & \vdots \\
d()d() & d(a_1)d() & d(a_1,a_2)d() & \dots & d(a_1,\dots,a_{n-1})d()
\end{array}\right), \]
where we define $d()=1$.  In other words, $A_{p/q}$ is symmetric, and for all $i\leq j$ the $(i,j)$th entry of $A_{p/q}$ is the positive integer $d(a_1,\dots,a_{i-1})d(a_{j+1},\dots,a_n)$.
\end{proposition}

\begin{proof}
For $i\leq j$, it is straightforward (though extremely tedious) to check that the cofactor of $M_{p/q}$ corresponding to position $(i,j)$ is
\[ C_{ij} = C_{ji} = (-1)^{i-1}d(a_1,\dots,a_{i-1}) \cdot (-1)^{n-j}d(a_{j+1},\dots,a_n). \]
(In other words, $C_{ij}$ is the product of the determinants of the tridiagonal matrices with $a_1,\dots,a_{i-1}$ and $a_{j+1},\dots,a_n$ on their diagonals and entries above and below the diagonal equal to $1$.)  The corresponding entry of $M^{-1}_{p/q}$ is therefore
\[ c_{ij} = (-1)^{i+j} \cdot \frac{(-1)^{i-1+n-j}|C_{ij}|}{\det(M_{p/q})} = \frac{(-1)^{n+1}|C_{ij}|}{(-1)^n p} = -\frac{d(a_1,\dots,a_{i-1})d(a_{j+1},\dots,a_n)}{p}, \]
so the entries of $A_{p/q} = -pM^{-1}_{p/q}$ are exactly as claimed.
\end{proof}

In particular, all of the entries $c'_{ij}$ of $A_{p/q}$ are positive integers, and so the corresponding value of $\rv^T A_{p/q} \rv$ is
\[ \sum_{i,j=1}^n c'_{ij}r_i r_j \leq \sum_{i,j=1}^n c'_{ij} |r_i r_j| \leq \sum_{i,j=1}^n c'_{ij}(|a_i|-2)(|a_j|-2) \]
with equality if and only if $r_i r_j = (|a_i|-2)(|a_j|-2)$ for all $i,j$.  We conclude that $\rv^T A_{p/q} \rv$ is maximized (and hence $d_3(\xi_{\rv})$ is minimized, as seen in equation \eqref{eq:d3-formula-apq}) exactly when all products $r_i r_j$ are nonnegative and as large as possible, i.e.\ when
\[ \rv = \langle r_1, r_2, \dots, r_n\rangle  = \pm \langle |a_1|-2, |a_2|-2, \dots, |a_n|-2\rangle. \]
These correspond to the canonical tight contact structure of Remark~\ref{rem:xi-can} and its conjugate, so we have shown the following.

\begin{proposition}
\label{prop:d3-xican-minimal}
If $\xi$ is a tight contact structure on $L(p,q)$, then $d_3(\xi) \geq d_3(\xican)$ with equality if and only if $\xi$ is either $\xican$ or its conjugate.
\end{proposition}

\begin{remark}\label{rmk:large-d3-not-reducible}
Proposition~\ref{prop:reducible-d3-value} says that any tight contact structure $\xi$ on $L(p,q)$ arising from a reducible Legendrian surgery has $d_3(\xi) = -\frac{r^2+p}{4p} \leq -\frac{1}{4}$.  It thus follows from Proposition~\ref{prop:d3-xican-minimal} that if $d_3(\xican) > -\frac{1}{4}$, then $L(p,q)$ cannot be a summand of a reducible Legendrian surgery.
\end{remark}

\subsection{A recurrence relation for the minimal $d_3$ invariant}
\label{ssec:recurrence}

Let $-\frac{p}{q} = [a_1,\dots,a_n]$ with $a_i\leq -2$ for all $i$, and define $f(p/q) = \rv^T A_{p/q} \rv$, where $\rv = \langle |a_1|-2,\dots,|a_n|-2 \rangle$ and $A_{p/q}$ is the matrix defined in the previous subsection whose entries are described by Proposition~\ref{prop:a-pq-entries}.  We note that $f(p/q)$ is integer-valued and that
\begin{equation}
\label{eq:d3-formula-fpq}
d_3(\xican) = \frac{-\frac{1}{p}f(\frac{p}{q}) + n-2}{4}
\end{equation}
according to equation \eqref{eq:d3-formula-apq}.  In this subsection we will determine a recurrence relation for $f(p/q)$ (Proposition~\ref{prop:f-recurrence-eqn}).  We begin with two examples.  

\begin{example}
\label{ex:f-n-1}
If $n=1$ then $\frac{p}{q}$ is an integer and $a_1 = -p$.  We compute $A_{p/1} = \left(1\right)$, and so $f(p/1) = (p-2)^2$.  Since $p\geq 2$, we note for later that $\frac{1}{p}f(\frac{p}{1}) \leq p-2$ with equality if and only if $p=2$.
\end{example}

\begin{example}
\label{ex:f-n-2}
If $n=2$ then $-\frac{p}{q} = a_1 - \frac{1}{a_2} = \frac{a_1a_2-1}{a_2}$, so $p=a_1a_2-1$ and $q=|a_2|$.  We compute
\begin{align*}
f(p/q) &=
\left(\begin{array}{cc} |a_1|-2 & |a_2|-2 \end{array}\right)
\left(\begin{array}{cc} |a_2| & 1 \\ 1 & |a_1| \end{array}\right)
\left(\begin{array}{c} |a_1|-2 \\ |a_2|-2 \end{array}\right) \\
&= |a_2|(|a_1|-2)^2 + 2(|a_1|-2)(|a_2|-2) + |a_1|(|a_2|-2)^2
\end{align*}
and so multiplying by $|a_2|$ produces
\[ |a_2|f(p/q) = \left[ |a_2|(|a_1|-2) + (|a_2|-2) \right]^2 + (|a_1||a_2|-1)(|a_2|-2)^2, \]
or $f(p/q) = \frac{1}{q}\left((p-q-1)^2 + p(q-2)^2\right)$.  Equivalently, $\frac{1}{p}f(\frac{p}{q}) = \frac{(p-q-1)^2}{pq} + \frac{1}{q}f(\frac{q}{1})$.
\end{example}

Now suppose that $n\geq 3$, i.e.\ that the continued fraction of $-\frac{p}{q}$ has length at least $3$.  Then if we let $r_i = |a_i|-2$ for all $i$, we have by Proposition~\ref{prop:a-pq-entries} that $f(p/q)$ is equal to
\begin{equation}\label{eq:f-matrix-formula} 
\left(\begin{array}{c} r_1 \\ r_2 \\ \vdots \\ r_n \end{array}\right)^T
\left(\begin{array}{cccc}
d()d(a_2,\dots,a_n) & d()d(a_3,\dots,a_n) & \dots & d()d() \\
d()d(a_3,\dots,a_n) & d(a_1)d(a_3,\dots,a_n) & \dots & d(a_1)d() \\
\vdots & \vdots & \ddots & \vdots \\
d()d() & d(a_1)d() & \dots & d(a_1,\dots,a_{n-1})d()
\end{array}\right)
\left(\begin{array}{c} r_1 \\ r_2 \\ \vdots \\ r_n \end{array}\right)
\end{equation}
where we recall that $d()=1$ by definition.  When we expand this product into $n^2$ terms, we can separate them into two groups: the $2n-1$ terms coming from the top row and left column of $A_{p/q}$, whose sum we denote by $S_1$, and the $(n-1)^2$ terms from the bottom right $(n-1)\times (n-1)$ submatrix, which sum to $S_2$.  The terms in the first group sum to
\begin{equation}\label{eq:S1-def}
 S_1 = r_1 \left(r_1 d(a_2,\dots,a_n) + 2\sum_{i=2}^n r_i d(a_{i+1},\dots,a_n)\right), 
\end{equation}
which we can simplify using the following lemma.

\begin{lemma}
\label{lem:d-recursion}
For $k \geq 2$ and integers $b_1,\dots,b_k \leq -2$ we have 
\[ d(b_1,\dots,b_k) = |b_1|d(b_2,\dots,b_k) - d(b_3,\dots,b_k). \]
\end{lemma}

\begin{proof}
Write $-\frac{p_i}{q_i} = [b_i,\dots,b_k]$ for $i=1,2,3$, so that $p_i = d(b_i,\dots,b_k)$.  Then we know from \eqref{eq:numerator-recursion} that
\[ -\frac{p_1}{q_1} = -\frac{|b_1|p_2 - q_2}{p_2} \]
so that $p_1 = |b_1|p_2 - q_2$ and $q_1 = p_2$, and similarly $q_2 = p_3$.  We conclude that $p_1 = |b_1|p_2-p_3$, as desired.
\end{proof}

Let $p_i = d(a_i,\dots,a_n)$ for $1 \leq i \leq n+1$, with $p_{n+1}=1$ and $p_n = |a_n|$, and recall that $n\geq 3$.  By Lemma~\ref{lem:d-recursion},  since $r_i = |a_i|-2$,  we have $r_{i-1}p_{i} = p_{i-1} - 2p_{i} + p_{i+1}$ for $2 \leq i \leq n-1$ and $r_{n}d() = r_n = |a_n|-2$.  Applying this to \eqref{eq:S1-def}, we obtain a telescoping sum
\begin{align*}
\frac{S_1}{r_1} &= (p_1-2p_2+p_3) + 2\left(\sum_{i=2}^{n-1} p_i-2p_{i+1}+p_{i+2}\right) + 2(p_n-2)(1) \\
&= (p_1 - 2p_2 + p_3) + 2(p_2 - p_3 - 1) \\
&= p_1 - p_3 - 2.
\end{align*}
In particular 
\begin{equation}\label{eq:S1}
S_1 = (|a_1|-2)(p_1-p_3-2) 
\end{equation}
for $n \geq 3$.

To evaluate the sum $S_2$ of the terms coming from the bottom right $(n-1) \times (n-1)$ submatrix of $A_{p/q}$ in \eqref{eq:f-matrix-formula}, we use Lemma~\ref{lem:d-recursion} to replace each factor of the form $d(a_1,\dots,a_{i-1})$ with $i\geq 3$.  Each term in the bottom right submatrix of $A_{p/q}$ can be rewritten as
\[ d(a_1,\dots,a_{i-1})d(a_{j+1},\dots,a_n) = \left(|a_1|d(a_2,\dots,a_{i-1}) - d(a_3,\dots,a_{i-1})\right)d(a_{j+1},\dots,a_n) \]
if $i \geq 3$, and otherwise if $i=2$ then 
\[ d(a_1,\dots,a_{i-1})d(a_{j+1},\dots,a_n) = |a_1|d()d(a_{j+1},\dots,a_n). \]
The contribution to $S_2$ from the $(i,j)$th or $(j,i)$th entry of $A_{p/q}$ for $2 \leq i \leq j \leq n$ is therefore one of
\begin{alignat*}{3}
&|a_1|\cdot \left(r_ir_j d(a_2,\dots,a_{i-1})d(a_{j+1},\dots,a_n)\right) &&- \left(r_ir_j d(a_3,\dots,a_{i-1})d(a_{j+1},\dots,a_n)\right), &\ \ \ &i\geq 4;\\
&|a_1|\cdot \left(r_ir_j d(a_2,\dots,a_{i-1})d(a_{j+1},\dots,a_n)\right) &&- \left(r_ir_j d()d(a_{j+1},\dots,a_n)\right), &\ \ \ &i=3;\\
&|a_1|\cdot \left(r_ir_j d()d(a_{j+1},\dots,a_n)\right), && &\ \ \ &i=2.
\end{alignat*}
Let $-\frac{p_2}{q_2} = [a_2,\dots,a_n]$ and $-\frac{p_3}{q_3} = [a_3,\dots,a_n]$.  Summing the above expressions over all $2 \leq i,j \leq n$, the terms with a factor of $|a_1|$ are precisely those appearing in $|a_1|f(p_2/q_2)$ while those without are precisely the terms in $f(p_3/q_3)$, so we conclude that
\begin{equation}\label{eq:S2}
S_2 = |a_1|f(p_2/q_2) - f(p_3/q_3).
\end{equation}

\begin{lemma}
\label{lem:f-difference-eqn}
Let $n\geq 3$, and write $-\frac{p_i}{q_i}=[a_i,\dots,a_n]$ for $1 \leq i \leq n$ for some sequence $\{a_i\}$.  Then
\[ q_1 f(p_1/q_1) - (p_1-q_1-1)^2 - p_1 f(p_2/q_2) = q_2f(p_2/q_2) - (p_2-q_2-1)^2 - p_2 f(p_3/q_3). \]
\end{lemma}

\begin{proof}
By \eqref{eq:S1} and \eqref{eq:S2}, we have 
\[ f(p_1/q_1) = S_1 + S_2 =  (|a_1|-2)(p_1-p_3-2) + |a_1|f(p_2/q_2) - f(p_3/q_3). \]
Now we use the identity $p_1 = |a_1|p_2 - p_3$ of \eqref{eq:numerator-recursion}, or equivalently $|a_1| = \frac{p_1+p_3}{p_2}$, and multiply both sides by $p_2$ to get
\[ p_2 f(p_1/q_1) = (p_1-2p_2+p_3)(p_1-p_3-2) + (p_1+p_3)f(p_2/q_2) - p_2f(p_3/q_3). \]
The first term on the right side equals the difference of squares $(p_1-p_2-1)^2 - (p_2-p_3-1)^2$, so after replacing it with this difference and using the identities $q_1 = p_2$ and $q_2 = p_3$, we have
\[ q_1 f(p_1/q_1) = (p_1-q_1-1)^2 - (p_2-q_2-1)^2 + (p_1+q_2)f(p_2/q_2) - p_2 f(p_3/q_3). \]
Rearranging both sides completes the proof.
\end{proof}

With this, we are ready to establish the desired recurrence relation for $f(p/q)$.  
\begin{proposition}
\label{prop:f-recurrence-eqn}
We have $f(p/q)=(p-2)^2$ if $\frac{p}{q}$ is an integer, i.e.\ if $q=1$, and otherwise
\[ \frac{1}{p}f\left(\frac{p}{q}\right) = \frac{(p-q-1)^2}{pq} + \frac{1}{q}f\left(\frac{q}{\lceil\frac{p}{q}\rceil q-p}\right). \]
\end{proposition}

\begin{proof}
In the case $\frac{p}{q} \in \ZZ$, this is Example~\ref{ex:f-n-1}, so assume that $\frac{p}{q} \not\in \ZZ$; then we can write $-\frac{p}{q}=[a_1,\dots,a_n]$, $n\geq 2$, and define $-\frac{p_i}{q_i} = [a_i,\dots,a_n]$ for all $i$.  (Note that $\frac{p}{q}=\frac{p_1}{q_1}$.)  Since $a_1 = \lfloor -\frac{p}{q} \rfloor$ we have $|a_1| = \lceil\frac{p}{q}\rceil$, and then $\frac{p}{q} = -(a_1 - \frac{1}{-p_2/q_2}) = \frac{|a_1|p_2-q_2}{p_2}$ implies that $\frac{p_2}{q_2} = \frac{q}{\lceil\frac{p}{q}\rceil q-p}$.

By applying Lemma~\ref{lem:f-difference-eqn} a total of $n-2$ times, we conclude that the quantity $\delta = q_1f(p_1/q_1)-(p_1-q_1-1)^2-p_1f(p_2/q_2)$ satisfies 
\[ \delta = q_{n-1}f(p_{n-1}/q_{n-1}) - (p_{n-1}-q_{n-1}-1)^2 - p_{n-1} f(p_n/q_n). \]
But we can evaluate these terms directly since $-\frac{p_{n-1}}{q_{n-1}}$ has a continued fraction of length $2$: we know from Example~\ref{ex:f-n-2} and because $-\frac{p_n}{q_n} = -\frac{q_{n-1}}{1}$ that 
\[ \frac{\delta}{p_{n-1}q_{n-1}} = \frac{1}{p_{n-1}}f\left(\frac{p_{n-1}}{q_{n-1}}\right) - \left[\frac{(p_{n-1}-q_{n-1}-1)^2}{p_{n-1}q_{n-1}} + \frac{1}{q_{n-1}}f\left(\frac{q_{n-1}}{1}\right)\right] = 0, \]
hence $\delta=0$.  Thus, $q_1f(p_1/q_1) = (p_1-q_1-1)^2 + p_1f(p_2/q_2)$, and dividing both sides by $p_1q_1$ produces the desired recurrence.  
\end{proof}

\subsection{A second proof of Theorem~\ref{thm:main}}

We recall from Proposition~\ref{prop:decomp} that any reducible surgery on $K$ of slope less than $\maxtb(K)$ has slope less than $-1$.

\begin{proposition}
\label{prop:reducible-d3-bound}
Let $K$ be a knot with $\maxtb(K) \geq 0$.  Suppose that $S^3_{-p}(K)$ is reducible for some $p \geq 2$, and write $S^3_{-p}(K) = L(p,q) \# Y$.  Then there is a tight contact structure $\xi$ on $L(p,q)$ induced by a Legendrian surgery on $K$ which satisfies
\begin{equation}\label{eqn:d3-lens-bad}
d_3(\xi) \leq -\frac{p+1}{4}. 
\end{equation}
\end{proposition}
\begin{proof}
We can take a Legendrian representative of $K$ with $tb = 0$ and ensure $r \geq 1$ by orienting $K$ appropriately.  Then if we positively stabilize this knot $p-1$ times, we will get a Legendrian representative with $tb = 1-p$ and $r \geq p$, and according to Proposition \ref{prop:reducible-d3-value} the contact structure $\xi$ on $L(p,q)$ induced by surgery on this representative satisfies
\[ d_3(\xi) = -\frac{r^2+p}{4p} \leq -\frac{p^2+p}{4p} = -\frac{p+1}{4}. \qedhere \]
\end{proof}

\begin{theorem}
\label{thm:d3-bound-lens-spaces}
Let $\xi$ be a tight contact structure on $L(p,q)$, $p>q\geq 1$.  If $-\frac{p}{q}$ has continued fraction
\[ [a_1,a_2,\dots,a_n] := a_1 - \frac{1}{a_2 - \frac{1}{\dots - \frac{1}{a_n}}}, \]
where $a_i \leq -2$ for all $i$, then $d_3(\xi) \geq \frac{-p+2n-1}{4}$.
\end{theorem}

\begin{proof}
We recall from equation \eqref{eq:d3-formula-apq} that $d_3(\xican) = \frac{1}{4}\big(-\frac{1}{p}f(\frac{p}{q})+n-2\big)$, so $d_3(\xican) \geq \frac{-p+2n-1}{4}$ if and only if $\frac{1}{p}f(\frac{p}{q}) \leq p-n-1$.  Moreover, it suffices to prove this last inequality, since Proposition~\ref{prop:d3-xican-minimal} says that $d_3(\xi) \geq d_3(\xican)$ for all tight contact structures $\xi$ on $L(p,q)$.  In Example~\ref{ex:f-n-1} we saw that this is true when $n=1$, so we will assume $n\geq 2$ and proceed by induction on $n$.

The recurrence of Proposition~\ref{prop:f-recurrence-eqn} says that
\begin{equation}\label{eq:f-p-q-recurrence} \frac{1}{p}f\left(\frac{p}{q}\right) = \frac{(p-q-1)^2}{pq} + \frac{1}{q}f\left(\frac{q}{\lceil\frac{p}{q}\rceil q-p}\right), 
\end{equation}
and we know that $p \geq q+1$ and so $0 \leq \frac{p-q-1}{pq} < \frac{p}{pq} \leq 1$.  Hence
\[ \frac{(p-q-1)^2}{pq} = (p-q-1)\cdot \frac{p-q-1}{pq} \leq p-q-1. \]
Now $-\frac{q}{\lceil\frac{p}{q}\rceil q-p}$ has a continued fraction of length $n-1$, so by assumption $\frac{1}{q}f\big(\frac{q}{\lceil\frac{p}{q}\rceil q-p}\big) \leq q-n$ and it follows from \eqref{eq:f-p-q-recurrence} that $\frac{1}{p}f\left(\frac{p}{q}\right) \leq (p-q-1) + (q-n) = p - n - 1$, as desired.
\end{proof}

In particular, if $\maxtb(K) \geq 0$ and $S^3_{-p}(K)$ is reducible for some $-p \leq \maxtb(K)-1$ then the induced tight contact structure $\xi$ on the $L(p,q)$ summand must satisfy $d_3(\xi) \geq -\frac{p-1}{4}$ by Theorem~\ref{thm:d3-bound-lens-spaces}, since $n\geq 1$.  Thus we have $d_3(\xi) >  -\frac{p+1}{4}$, but this contradicts Proposition~\ref{prop:reducible-d3-bound}.  We conclude once again that if $\maxtb(K)\geq 0$, then any reducible surgery on $K$ must have slope $\maxtb(K)$ or greater, which completes the second proof of Theorem~\ref{thm:main}.  

\begin{remark}
The condition $\maxtb(K) \geq 0$ in Proposition~\ref{prop:reducible-d3-bound} can be relaxed slightly as follows: if $K$ has a Legendrian representative with $tb(K)+|r(K)| \geq 1$ (or equivalently $\geq 0$, since the left side is always odd), then the conclusion still holds for all slopes $-p \leq tb(K)-1$ even if $tb(K) < \maxtb(K)$ or $\maxtb(K) < 0$.  Thus we can actually rule out sufficiently negative reducible surgeries on any knot $K$ with $\maxsl(K) \geq 1$, making this proof of Theorem~\ref{thm:main} slightly stronger than the proof given in Section~\ref{ssec:proof-main}.
\end{remark}


\section{Reducible Legendrian surgeries for $\maxtb < 0$}\label{sec:negative-tb}

In this section we will prove Theorem~\ref{thm:tb-negative-thm}, which we restate here for convenience.

\begin{theorem}
\label{thm:large-negative-surgeries}
Suppose that $\maxtb(K) = -\tau$ for some $\tau > 0$, and define $t=\tau$ if $\tau$ is odd and $t=\tau-1$ if $\tau$ is even.  Suppose also that $-p$-surgery on $K$ is reducible for some $p>\tau$.
\begin{itemize}
\item If $p > 2t-3$, then $S^3_{-p}(K) = S^3_{-p}(U) \# Y$.
\item If $p=2t-3$ and $S^3_{-p}(K)$ does not have the form $S^3_{-p}(U) \# Y$, then we must have $(\tau,p)=(6,7)$ and $S^3_{-p}(K) \cong S^3_{-7}(T_{2,-3}) \# Y$, with trace diffeomorphic to a boundary sum $X \natural Z$, where $X$ is the trace of $-7$-surgery on $T_{2,-3}$.
\item  Moreover, whenever the reducible surgery has the form $S^3_{-p}(K) = S^3_{-p}(U) \# Y$, the trace of the surgery is diffeomorphic to $D_{-p} \natural Z$, where $D_{-p}$ is the disk bundle over $S^2$ with Euler number $-p$.
\end{itemize}
In each case, $Y$ is an irreducible homology sphere bounding the contractible Stein manifold $Z$.
\end{theorem}

\begin{proof}
Since we assume $p > \tau$, we can apply Proposition~\ref{prop:decomp} to the Legendrian reducible surgery on $K$.  We note that $L(p,1) = S^3_{-p}(U)$, and so we begin by ruling out all $L(p,q)$ as possible summands where $q > 1$, when $p > 2t - 3$.  Note that since we require $q < p$, the condition $q > 1$ is equivalent to the continued fraction $-\frac{p}{q} = [a_1,\dots,a_n]$ having length $n \geq 2$.  

First, suppose there is an $i$ such that $a_i = -2$.  If $n \geq 3$ then we apply Proposition~\ref{prop:p-bound-ai-2} below to see that $p \leq 2t-4$.  If instead we have $n=2$, then Proposition~\ref{prop:possible-tb-Lp2} says that the lens space summand is $L(7,2) \cong L(7,4)$ and $\tau$ is either $5$ or $6$, hence $p=7=2t-3$.  We return to this case shortly.  

In the remaining case we have $a_i \leq -3$ for all $i$, and this will require a closer examination of the $d_3$ invariants of tight contact structures on $L(p,q)$, but we will eventually prove in Proposition~\ref{prop:p-q-3-n-2} and Corollary~\ref{cor:p-bound-ai-large} (corresponding to $n=2$ and $n \geq 3$ respectively) that $p \leq 2t-3$ as desired, with equality only if $n=2$ and $L(p,q) \cong L(p,4)$.  If equality occurs, then we must also have $p\equiv 3 \pmod{4}$ since $t$ is odd and $p=2t-3$.  Thus if $q>1$ we conclude that $p \leq 2t-3$, with equality only if $L(p,q) \cong L(4k+3,4)$ for some $k \geq 1$.

Now suppose that $p=2t-3$ and $L(p,q) \cong L(p,4)$.  In order to achieve $p = 2t - 3$, by Propositions~\ref{prop:possible-tb-Lp2} and \ref{prop:p-q-3-n-2}, there must be a Legendrian representative of $K$ with $(tb,r)=(1-p, \frac{p-5}{2})$ which induces the contact structure $\xican$ or its conjugate on $L(p,4)$ as explained in Example~\ref{ex:L-4k+3-4} below; and if $\tau \neq 6$ then this representative is stabilized.  Performing Legendrian surgery on this representative, the trace is a Stein manifold $(X,J) \cong (W,J_W) \natural (Z,J_Z)$, where $(W,J_W)$ is a Stein filling of $(L(p,4),\xican)$ up to conjugation, by Proposition~\ref{prop:decomp}.  Lemma~\ref{lem:fillings-L-p-4} then asserts that $W$ is symplectic deformation equivalent to the trace of Legendrian surgery on the $(2,-\frac{p-1}{2})$-torus knot with the same values of $tb$ and $r$.  Thus the symplectic homology $\sh(W)$ is equal to that of the trace of this Legendrian surgery, and we will use this fact to get a contradiction for most values of $\tau$.\footnote{We make use of several results from \cite{bee}, which does not claim to provide complete proofs in full generality.  However, since we are not interested in the actual value of $\sh(X)$ but only whether it vanishes or not, the available details will suffice for our purposes.}

Let $(\mathcal{A}_p,\partial_p)$ denote the Legendrian contact homology DGA over $\Q$ of the $(2,-\frac{p-1}{2})$-torus knot $T$ with $(tb,r) = (1-p,\frac{p-5}{2})$.  It is easy to see that a front diagram of this knot admits an ungraded normal ruling and hence an ungraded augmentation \cite{fuchs,leverson}, which is an ungraded DGA morphism $f:(\mathcal{A}_p,\partial_p) \to (\Q, 0)$; this implies that $1 \not\in \operatorname{im}(\partial_p)$, and so the homology group $\lhho(T)$ of \cite[Section 4.5]{bee} is nonzero.  Then Bourgeois--Ekholm--Eliashberg \cite[Corollary 5.7]{bee} proved that $\sh(W) \cong \lhho(T)$, which is nonzero, and we have an isomorphism of rings $\sh(X) \cong \sh(W) \times \sh(Z) \neq 0$ as well by \cite[Theorem 1.11]{cieliebak} and \cite[Theorem 2.20]{mclean}.  But if $\tau \neq 6$, then as mentioned above, $X$ is the trace of surgery on a stabilized Legendrian knot, hence $\sh(X)=0$ by \cite[Section~7.2]{bee}.  We conclude that if $p=2t-3$ then we must have $(\tau,p)=(6,7)$, and $W$ is the trace of $-7$-surgery on the left handed trefoil.

Finally, assuming that $S^3_{-p}(K) = L(p,1) \# Y$, we let $(W,J_W)$ denote the induced Stein filling of $L(p,1)$.  The symplectic fillings of $L(p,1)$ have been completely classified, by McDuff \cite{mcduff} and Hind \cite{hind} in the universally tight case and by Plamenevskaya--van Horn-Morris \cite{plamenevskaya-vhm} in the virtually overtwisted case: up to blowing up, the fillings are all deformation equivalent to either the fillings described by Theorem~\ref{thm:lens-space-classification} (i.e., attaching a Weinstein $2$-handle to $B^4$ along a topological unknot with $tb=1-p$, the result of which is diffeomorphic to $D_{-p}$), or a rational homology ball bounded by $L(4,1)$.  Note that $W$ has intersection form $\langle -p \rangle$ and so it cannot be diffeomorphic to a blow-up, since an exceptional sphere would have self-intersection $-1$, or to a rational homology ball.  Thus $(W,J_W)$ must be deformation equivalent to $D_{-p}$ with a corresponding Stein structure coming from Theorem~\ref{thm:lens-space-classification}.  
\end{proof}

The rest of the current section is devoted to establishing the results claimed in the proof of Theorem~\ref{thm:large-negative-surgeries}.  We will study Stein fillings of $(L(p,4),\xican)$ momentarily, and then in the following subsections we will show that the various lens spaces $L(p,q)$, $q>1$, cannot be summands of reducible surgeries on $K$ whose slopes are sufficiently negative with respect to $\maxtb(K)$.  We divide this into two cases, each of which requires a different strategy, based on the continued fraction $-\frac{p}{q} = [a_1,\dots,a_n]$: in Section~\ref{ssec:lens-ai-2} we study the lens spaces for which $\max_i a_i = -2$, and in Section~\ref{ssec:lens-ai-3} we deal with the lens spaces such that $a_i \leq -3$ for all $i$.
 
\begin{lemma}
\label{lem:fillings-L-p-4}
If $p=4k+3$ for some $k \geq 1$, then any Stein filling $W$ of $(L(p,4), \xican)$ with $b_2(W)=1$ and intersection form $\langle -p\rangle$ is symplectic deformation equivalent to the trace of Legendrian surgery on the Legendrian $(2,-\frac{p-1}{2})$-torus knot with $tb=1-p$ and $r=\frac{p-5}{2}$.
\end{lemma}

\begin{proof}
Bhupal and Ono \cite{bhupal-ono} proved that diffeomorphic fillings of the canonical contact structure on a lens space must be deformation equivalent, so it suffices to prove that $W$ is unique up to diffeomorphism.  We use Lisca's classification of fillings of $(L(p,4), \xican)$ up to diffeomorphism and blow-up \cite{lisca-fillings_announce, lisca-fillings_proof}; note that once again $W$ cannot be diffeomorphic to a blow-up.  As explained in Example~\ref{ex:L-4k+3-4} below, the trace of Legendrian surgery on a representative of $T_{2,-\frac{p-1}{2}}$ with $(tb,r)=(1-p, \frac{p-5}{2})$ gives a Stein filling of $(L(p,4),\xican)$ or its conjugate, so we need to check that $W$ is diffeomorphic to this trace.

In Lisca's notation, we have a continued fraction $\frac{p}{p-4} = [2,2,\dots,2,3,2,2]$ of length $m=\frac{p+5}{4}$, so the minimal symplectic fillings of $L(p,4)$ have the form $W_{p,4}(n_1,\dots,n_m)$ where $n_i \leq 2$ for $i \neq m-2$, $n_{m-2} \leq 3$, and the sequence $[n_1,\dots,n_m]$ is obtained from $[0]$ by a ``blow-up'' procedure.  Each blow-up sending $[n_1,\dots,n_l]$ to $[1,n_1+1,\dots,n_l]$ or $[n_1,\dots,n_l+1,1]$ increases $\sum n_i$ by 2 while the blow-up $[n_1,\dots,n_i+1,1,n_{i+1}+1,\dots,n_l]$ increases it by 3; since we must blow up $m-1$ times to get a sequence of length $m$, and the condition $b_2=1$ is equivalent to $\sum n_i = 2m-1$, we must perform the latter operation exactly once.  In addition, we must ensure that at each step at most one entry $n_i$ exceeds 2, in which case it equals 3.  Applying the former operation some number of times produces $[1,1]$ or $[1,2,2,\dots,2,1]$; then the latter can produce either $[2,1,2]$ or $[1,\dots,3,1,2]$ or its reverse, in which every omitted entry is a 2; and then repeating the former operation a nonnegative number of times leaves us with a sequence of the same form.  Thus $W$ must be diffeomorphic to either
\begin{align*}
W_{7,4}(2,1,2), & & W_{15,4}(1,2,3,1,2) \cong W_{15,4}(2,1,3,2,1) &,
\end{align*}
or $W_{p,4}(1,\dots,3,1,2)$ for $p \neq 7,15$, and so it is uniquely determined by $p$.
\end{proof}

We will use the following terminology throughout this section.

\begin{definition}
The integer $r$ is a \emph{rotation number for $L(p,q)$} if there is a tight contact structure $\xi$ on $L(p,q)$ such that $d_3(\xi) = -\frac{r^2+p}{4p}$.
\end{definition}

The reason for the terminology is that a reducible Legendrian surgery on a Legendrian representative of $K$ with $L(p,q)$ as a summand, where $(tb(K), r(K)) = (1-p, r)$, induces a tight contact structure $\xi$ on $L(p,q)$ with $d_3(\xi) = -\frac{r^2+p}{4p}$ by Proposition \ref{prop:reducible-d3-value}, hence $r$ is a rotation number for $L(p,q)$.  We will thus study the set of rotation numbers for each $L(p,q)$, $q > 1$, in order to produce bounds on $p$ in terms of $\maxtb(K)$.

\begin{example}
\label{ex:L-4k+3-4}
If $p=4k+3$ then $L(p,4)$ is the result of $-p$-surgery on the torus knot $T_{2,-(2k+1)}$, i.e.\ the $(2,-\frac{p-1}{2})$ cable of the unknot \cite{moser}.  This can be realized as a Legendrian surgery since $\maxtb(T_{2,-(2k+1)}) = -4k-2 = 1-p$, and in fact every odd number from $-(2k-1)$ to $2k-1 = \frac{p-5}{2}$ is the rotation number of some $tb$-maximizing representative of $T_{2,-(2k+1)}$ \cite{etnyre-honda-knots}.  We conclude that if $p \equiv 3\pmod{4}$, then every odd number $r$ with $|r| \leq \frac{p-5}{2}$ is a rotation number for $L(p,4)$.  Moreover, since $-\frac{p}{4} = [-(k+1),-4]$ we can use Example~\ref{ex:f-n-2} and \eqref{eq:d3-formula-fpq} to compute $d_3(\xican) = -\frac{1}{4p}\left((\frac{p-5}{2})^2 + p\right)$, so the rotation number $r=\frac{p-5}{2}$ comes from $\xican$.  Further, by Proposition~\ref{prop:d3-xican-minimal}, $\xican$ and its conjugate are the only tight contact structures on $L(p,4)$ which can produce this rotation number.  
\end{example}

\subsection{Lens spaces with $a_i = -2$ for some $i$}
\label{ssec:lens-ai-2}

In this section we prove part of Theorem \ref{thm:large-negative-surgeries}, namely that $L(p,q)$ summands cannot occur under the hypotheses of the theorem if some entry in the continued fraction for $-\frac{p}{q}$ is $-2$.  This will follow in most cases by counting the possible tight contact structures on $L(p,q)$ coming from a reducible surgery on the knot $K$.

\begin{proposition}
\label{prop:p-bound-ai-2}
Let $K$ be a knot with $\maxtb(K) = -\tau$, $\tau \geq 1$, and suppose that $S^3_{-p}(K) = L(p,q) \# Y$ for some $p>\tau$ and $q>1$.  If $-\frac{p}{q} = [a_1,\dots,a_n]$ with $n \geq 3$, and $a_i=-2$ for some $i$, then 
\[ p \leq 2t - n \]
where $t=\tau$ if $\tau$ is odd and $t=\tau-1$ if $\tau$ is even, and this inequality is strict for $n=3$.
\end{proposition}

\begin{proof}
Choose a Legendrian representative of $K$ with $(tb,r)=(-\tau,r_0)$ and $r_0 \geq 0$.  By stabilizing $p-\tau-1$ times with different choices of sign, we get Legendrian representatives with $tb=1-p$ and
at least $p-\tau$ different values of $r$.  If $\tau$ is even, then we can take $r_0 \geq 1$ since $tb+r$ is always odd, and then we can also negatively stabilize a representative with $(tb,r)=(-\tau,-r_0)$ a total of $p-\tau-1$ times to get a $(p - \tau + 1)$th value of $r$.  Therefore, independent of the parity of $\tau$, there are at least $p-t$ different values of $r$ when $tb=1-p$.  By Proposition \ref{prop:reducible-d3-value} each value of $-\frac{1}{4p}(r^2+p)$ must be $d_3(\xi)$ for some tight contact structure on $L(p,q)$, so by counting values of $r^2$ we see that the set
\[ \left\{ d_3(\xi) \mid \xi \in \operatorname{Tight}(L(p,q)) \right\} \]
has at least $\left\lceil \frac{p-t}{2} \right\rceil$ elements.  Moreover, none of these is the $d_3$ invariant of a self-conjugate contact structure $\xi_0$ (if one even exists), since we would have $d_3(\xi_0) = \frac{n-2}{4} \geq 0$ by \eqref{eq:d_3-lens-gompf}, contradicting Remark~\ref{rmk:large-d3-not-reducible}.

Since each value of $d_3(\xi)$ above is obtained by at least one conjugate pair of tight contact structures on $L(p,q)$, the number of tight contact structures on $L(p,q)$ is at least 
\[ 2\left\lceil\frac{p-t}{2}\right\rceil \geq p-t. \]
But by the case $m=2$ of Proposition~\ref{prop:count-xi}, $L(p,q)$ has at most $\frac{p-n+1}{2}$ tight contact structures with equality if and only if $n=2$ or $p=q+1$; if $p=q+1$ then $-\frac{p}{q} = [-2,-2,\dots,-2]$ and the unique tight contact structure on $L(p,q)$ is self-conjugate, so this does not occur.  Thus the number of tight contact structures on $L(p,q)$ is actually at most $\frac{p-n}{2}$, which implies
\[ p - t \leq \frac{p-n}{2} \]
or equivalently $p \leq 2t-n$.

If $n=3$ and we have the equality $p=2t-n$, then the number of tight contact structures on $L(p,q)$ must have been exactly $\frac{p-n}{2} = \frac{p-3}{2}$.  Writing $-\frac{p}{q} = [a_1,a_2,a_3]$, we can assume, after possibly replacing $[a_1,a_2,a_3]$ with $[a_3,a_2,a_1]$ as in the proof of Proposition~\ref{prop:count-xi}, that either $a_2=-2$ or $a_3=-2$, hence if $-\frac{q}{r} = [a_2,a_3]$ then Proposition~\ref{prop:count-xi} says that $L(q,r)$ has exactly $\frac{q-1}{2}$ tight contact structures.  Using Theorem~\ref{thm:lens-space-classification}, the number of tight contact structures on $L(p,q)$ is therefore
\[ (|a_1|-1)\left(\frac{q-1}{2}\right) = \frac{|a_1|q-q - (|a_1|-1)}{2} = \frac{p - (q-r) - (|a_1|-1)}{2}, \]
where the second equality follows from \eqref{eq:numerator-recursion}.
This number equals $\frac{p-3}{2}$ only if $(q-r,a_1)$ is either $(1,-3)$ or $(2,-2)$, hence $-\frac{p}{q}$ is $[-3,-2,-2]$ or $[-2,-2,-3]$.  Thus if $p=2t-3$ then $L(p,q) \cong L(7,3) \cong L(7,5)$, but then $d_3(\xican) = \frac{1}{7} > -\frac{1}{4}$, so by Remark~\ref{rmk:large-d3-not-reducible}, $L(7,3)$ cannot be a summand and we must have $p \leq 2t-4$ as claimed.
\end{proof}

Suppose that instead of being in the setting of Proposition~\ref{prop:p-bound-ai-2}, we have $n=2$, and $a_i = -2$ for some $i$.  Then $-\frac{p}{q}$ is either $[-(k+1),-2] = -\frac{2k+1}{2}$ or $[-2,-(k+1)] = -\frac{2k+1}{k+1}$ for some $k \geq 1$.  Hence we have $L(p,q) = L(2k+1,2) \cong L(2k+1,k+1)$.  In particular, $p$ must be odd when $n=2$ and $L(p,q) \cong L(p,2)$.  In this case, we have very strong restrictions both on $p$ and on $K$.  

\begin{proposition}
\label{prop:possible-tb-Lp2}
If $L(p,2)$ is a summand of a reducible Legendrian surgery on some knot $K$, then 
$p=7$, $\maxtb(K)$ is either $-5$ or $-6$, and $K$ has a Legendrian representative with $(tb,r)=(-6,1)$.  If $\maxtb(K)=-5$ then we can take this representative to be a stabilization.
\end{proposition}

\begin{proof}
Since $p$ must be odd, we write $p=2k+1$, $k\geq 1$, and $-\frac{p}{2} = [-k-1,-2]$.  By Theorem~\ref{thm:lens-space-classification}, the induced tight contact structure $\xi$ on $L(p,2)$ is the result of Legendrian surgery on a Hopf link whose components have $tb$ equal to $-k$ and $-1$ and rotation numbers $s\in \{-k+1,-k+3,\dots,k-1\}$ and $0$ respectively.  We compute $d_3(\xi) = -\frac{2s^2}{4p}$, and if $r$ is the rotation number for $L(p,2)$ corresponding to $\xi$ then $d_3(\xi) = -\frac{r^2+p}{4p}$ and so $r^2+p=2s^2$.  Then $r$ must be odd, so $r^2\equiv 1\pmod{8}$, and since $2s^2$ is either $0$ or $2\pmod{8}$ we have $p = 2s^2-r^2 \equiv \pm 1\pmod{8}$.  In particular, $p \geq 7$.  

Suppose that $p\neq 7$.  We begin with the case $p=9$ and observe that $s \in \{-3,-1,1,3\}$.  We cannot have $s=\pm 1$ since $r^2+9=2(\pm 1)^2$ has no solutions, so $s=\pm 3$.  Because $|a_1| = 5$, the induced contact structure on $L(9,2)$ is either $\xican$ or its conjugate.  Lisca \cite[Second Example]{lisca-fillings_announce} showed that $\xican$ has two symplectic fillings up to diffeomorphism and blow-up, and these satisfy $b_2(W_{9,2}(1,2,2,1)) = 2$ and $b_2(W_{9,2}(2,2,1,3)) = 0$ in Lisca's notation (using the continued fraction $\frac{9}{9-2} = [2,2,2,3]$).  Otherwise $p \geq 15$, and Kaloti \cite[Theorem~1.10]{kaloti} showed that for such $p$, every tight contact structure $\xi$ on $L(p,2)$ has a unique Stein filling up to symplectomorphism; since $\xi$ is presented as surgery on a Hopf link, this filling evidently has $b_2(X)=2$.  Thus in any case there is no Stein filling of $(L(p,2),\xi)$ with intersection form $\langle -p \rangle$, which contradicts Proposition~\ref{prop:decomp}.

We must therefore have $p=7$, and $r^2 + 7 = 2s^2$ with $s\in\{\pm 2, 0\}$ has only the solutions $(r,s)=(\pm 1,\pm 2)$, so $L(7,2)$ has rotation numbers precisely $\pm 1$.  If $K$ had a Legendrian representative with $(tb,r)=(-4,r_0)$, i.e., $\maxtb(K) \geq -4$, then it would have representatives with $tb=-6$ and $r \in \{r_0 - 2, r_0, r_0 + 2\}$, so $L(7,2)$ would actually have at least three rotation numbers, since $r_0$ is necessarily odd.  We conclude that $\maxtb(K) \leq -5$ as claimed, and the rotation numbers of the $tb=-6$ representatives of $K$ must be $\pm 1$ since these are the only rotation numbers for $L(7,2)$; if $\maxtb(K) = -5$ then it follows that any $tb$-maximizing representative has $r=0$ and thus positively stabilizes to a representative with $(tb,r)=(6,1)$.
\end{proof}

\subsection{Lens spaces with $a_i \leq -3$ for all $i$}
\label{ssec:lens-ai-3}

In this section we will complete the proof of Theorem \ref{thm:large-negative-surgeries} by studying lens spaces $L(p,q)$ such that every entry in the continued fraction for $-\frac{p}{q}$ is at most $-3$.  In this case, we will need to examine the $d_3$ invariants of tight contact structures on $L(p,q)$ more carefully.  The proof is divided into several cases depending on the length $n$ of the continued fraction.

\subsubsection{Case 1: $n=2$}

In this case it is not hard to explicitly determine $d_3(\xi)$ for any tight contact structure $\xi$ on $L(p,q)$, and using this we can restrict the set of rotation numbers for $L(p,q)$.

\begin{proposition}
\label{prop:p-q-3-n-2}
Suppose that a knot $K$ with $\maxtb(K)=-\tau$, $\tau > 0$, has a reducible Legendrian surgery with an $L(p,q)$ summand, where $-\frac{p}{q} = [-a,-b]$ for some $a,b \geq 3$.  Then $p \leq 2t-3$, where $t = \tau$ if $\tau$ is odd and $t=\tau-1$ if $\tau$ is even.  Moreover, if we have equality then $L(p,q) \cong L(p,4)$ and $K$ has a Legendrian representative with $(tb,r) = (1-p, \frac{p-5}{2})$, which is obtained by positively stabilizing a representative with $tb=\maxtb(K)$ a positive number of times.
\end{proposition}

\begin{proof}
We note that $p=ab-1$ and $q=b$, and since $L(ab-1,b) \cong L(ab-1,a)$ we can assume without loss of generality that $b \leq a$, hence $q \leq \sqrt{p+1}$.

Suppose that $r$ is a rotation number for $L(p,q)$ corresponding to the tight contact structure $\xi$.  Since $n = 2$, Proposition~\ref{prop:d3-xican-minimal}, the computation of Example~\ref{ex:f-n-2}, and \eqref{eq:f-p-q-recurrence} imply that
\[ -\frac{r^2+p}{4p} = d_3(\xi) \geq d_3(\xican) = -\frac{1}{4}\left(\frac{(p-q-1)^2 + p(q-2)^2}{pq}\right), \]
or upon multiplying by $-4p$ and rearranging,
\[ r^2 \leq \frac{(p-q-1)^2}{q} + p\left(\frac{(q-2)^2}{q} - 1\right). \]
We denote the right side of the above inequality, viewed as a function of $q$, as $\psi(q)$.  We see that $\psi$ has derivative $(p+1)\left(1-\frac{p+1}{q^2}\right)$, hence is decreasing on the interval $3 \leq q \leq \sqrt{p+1}$.  Thus $\psi$ is largest when $q$ is as small as possible, and $\psi(4) = \frac{(p-5)^2}{4}$, so we must have $-\frac{p-5}{2} \leq r \leq \frac{p-5}{2}$ except possibly when $q = 3$.
Moreover, if $q>4$ then this inequality is strict.

Suppose that $q \neq 3$ and that a Legendrian representative of $K$ with $tb = -\tau$ has rotation number $r_0 \geq 0$.  Then by positively stabilizing $p-\tau-1$ times we get a representative with $(tb,r) = (1-p, r_0 + p-\tau-1)$, hence this $r$ is a rotation number for $L(p,q)$.  In particular $r_0$ is at least 0 if $\tau$ is odd and 1 if $\tau$ is even, so we have $r \geq p-t-1$; and $r \leq \frac{p-5}{2}$ as explained above, so $p-t-1 \leq \frac{p-5}{2}$ or equivalently $p \leq 2t-3$, with equality only if $q=4$ and the above representative has $r=p-t-1=\frac{p-5}{2}$.  In particular, if there is equality, the lens space summand must be $L(p,4)$, as claimed.  Also, in case $p=2t-3$ we have $p = 4a-1 \geq 11$, since $p = ab - 1$ and we assume $a \geq 3$.  Hence $t \geq 7$; but then $p-\tau-1 \geq (2t-3)-(t+1)-1 > 0$, so this representative must actually be a stabilization.

If $q=3$ instead, then we only get the bound $|r| \leq \frac{p-5}{\sqrt{3}}$ above, but in fact we will see that $p \leq \tau + 2$ except possibly when $p=11$ and $7 \leq \tau \leq 10$.    Our strategy is as follows: supposing that $p > \tau+2 = 2-\maxtb(K)$, we can stabilize a $tb$-maximizing representative as needed to get a Legendrian representative of $K$ with $(tb(K), r(K)) = (3-p, r_0)$ for some $r_0$, which we then stabilize to get representatives with $tb=1-p$ and $r_0 \in \{r_0-2, r_0, r_0+2\}$, and so these three consecutive numbers of the same parity are rotation numbers for $L(p,q)$.  When $p\neq 11$ we will see that this cannot be the case, by determining when two consecutive numbers of the same parity can be rotation numbers for $L(p,3)$; it will follow that $\maxtb(K) \leq 2-p$, or equivalently $p \leq \tau+2$.  If instead $p=11$ then we have $L(11,3)\cong L(11,4)$, and the argument in the $q>3$ case above will apply to show that $p \leq 2t-3$ with equality only if $K$ has a representative with $(tb,r)=(-10,3)$.  Thus we may assume $p\neq 11$ from now on.  After we establish the bound $p \leq \tau + 2$, we will show that this implies $p \leq 2t-4$ unless the lens space is $L(8,3)$ and then complete the proof by analyzing the Stein fillings of $L(8,3)$.

By Theorem~\ref{thm:lens-space-classification}, each tight contact structure $\xi$ on $L(p,3)$ comes from Legendrian surgery on a Hopf link whose components have $tb$ equal to $1-a$ and $-2$, and rotation numbers $s$ and $\pm 1 = u$ respectively.  By \eqref{eq:d3-formula-apq}
\[ d_3(\xi) = -\frac{1}{4}\left(\frac{ 3s^2 + 2su + au^2}{p}\right) = -\frac{3s^2 \pm 2s + a}{4p}. \]
If $r$ is a rotation number for $L(p,q)$ corresponding to $\xi$, then
\[ r^2 + p = 3s^2 \pm 2s + a = \frac{(3s \pm 1)^2 + 3a-1}{3}, \]
and since $p=3a-1$ this is equivalent to $3r^2 + 2p = (3s\pm 1)^2$.  Thus $r$ can only be a rotation number for $L(p,3)$ if there exists some integer $s$ satisfying this equation.

As mentioned in the strategy above, suppose that both $k + 1$ and $k - 1$ are rotation numbers for $L(p,q)$, and $k\geq 0$ without loss of generality; write 
\begin{align}
\label{eq:s_+} 3(k+1)^2 + 2p &= s_+^2 \\
\label{eq:s_-} 3(k-1)^2 + 2p &= s_-^2
\end{align}
for some integers $s_\pm \geq 0$.  Subtracting \eqref{eq:s_-} from \eqref{eq:s_+}, we get $12k = s_+^2-s_-^2$, hence $s_+$ and $s_-$ have the same parity and we can write $s_\pm = c \pm \delta$ for some integers $c \geq \delta \geq 0$.  This gives $12k = 4c\delta$, or $3k = c\delta$, and so multiplying \eqref{eq:s_+} by 3 gives
\[ 3c^2 + 6c\delta + 3\delta^2 = 9k^2 + 18k + 9 + 6p = c^2\delta^2 + 6c\delta + 9 + 6p, \]
or equivalently $(c^2 - 3)(\delta^2 - 3) + 6p = 0$.  Since $p>0$ and $c \geq \delta$, this is impossible if $\delta \geq 2$.  Thus $(c,\delta)$ must be either $(\sqrt{2p+3}, 0)$ or $(\sqrt{3p+3}, 1)$.  Since $k=\frac{c\delta}{3}$, we conclude that numbers which differ by $2$ can only both be rotation numbers if they are $\pm 1$ or $\pm\left(\sqrt{\frac{p+1}{3}} \pm 1\right)$, hence if they are both nonnegative then they equal $\sqrt{\frac{p+1}{3}} \pm 1$.

In particular, suppose that $r_0-2, r_0, r_0+2$ are all rotation numbers, and $r_0 \geq 0$ without loss of generality.  Then by the above $r_0$ and $r_0+2$ must equal $\sqrt{\frac{p+1}{3}} \pm 1$.  Since $r_0-2$ and $r_0$ are another pair of rotation numbers which differ by 2, they must be either $\pm 1$ or $-\left(\sqrt{\frac{p+1}{3}}\pm 1\right)$; but the elements in the latter pair are both negative since $p>3$, whereas $r_0 \geq 0$, so $r_0$ and $r_0-2$ must equal $\pm 1$.  Thus we have $1 = r_0 = \sqrt{\frac{p+1}{3}}-1$ and $p=11$.  We conclude as explained above that if $p\neq 11$ then $p \leq \tau+2$.

The inequality $p \leq \tau+2$ implies that $p \leq 2t-4$ for all $\tau \geq 7$, so this leaves only $\tau \leq 6$, in which case $p \leq \tau+2 \leq 8$ and the only such lens space is $L(8,3)$.  In this case we have $d_3(\xican)=-\frac{1}{4}$ and so the induced contact structure must be $\xican$ or its conjugate, by Proposition~\ref{prop:d3-xican-minimal}, with a Stein filling $W$ having intersection form $\langle-8\rangle$ by Proposition~\ref{prop:decomp}.  Lisca \cite{lisca-fillings_announce} showed that $W$ must be diffeomorphic to a blow-up of one of two fillings, denoted $W_{8,3}(1,2,1)$ or $W_{8,3}(2,1,2)$ (note that $\frac{8}{8-3}=[2,3,2]$), and the first cannot occur since it has $b_2 = 2$.  The second is constructed from a diagram in which $2$-handles are attached to a pair of parallel $-1$-framed unknots; the cocores of these handles, together with the annulus they cobound, produce a sphere of self-intersection $-2$, and so the intersection form of this or any blow-up cannot be $\langle -8\rangle$.  We conclude that $L(8,3)$ cannot occur as a summand, completing the proof.
\end{proof}

\begin{remark}
\label{rem:small-lens-spaces}
We have seen in the proofs of Propositions~\ref{prop:possible-tb-Lp2} and \ref{prop:p-q-3-n-2} that $L(p,2)$ for $p\neq 7$ and $L(8,3)$ cannot be summands of reducible Legendrian surgeries, in both cases by examining their symplectic fillings.  We can rule out many other lens spaces $L(p,q)$ for $p$ small simply by computing $d_3(\xican)$ and appealing to Remark~\ref{rmk:large-d3-not-reducible}: on $L(p,p-1)$ we have $d_3(\xican) = \frac{p-3}{4} > -\frac{1}{4}$ for $p>2$, and on $L(7,3)$, $L(8,5)$, $L(9,4)$, and $L(10,3)$ we compute that $d_3(\xican)$ is equal to $\frac{1}{7}$, $\frac{1}{8}$, $\frac{7}{18}$, and $-\frac{1}{20}$ respectively.  Up to homeomorphism, this eliminates all lens spaces with $p\leq 10$ as possible summands except for $L(p,1)$ and $L(7,2) \cong L(7,4)$.
\end{remark}


\subsubsection{Case 2: $n \geq 3$}

In order to rule out $L(p,q)$ summands in Theorem \ref{thm:large-negative-surgeries} in this case, where $n \geq 3$ and $a_i \leq -3$ for all $i$, we will need to bound $d_3(\xican)$ carefully enough to restrict the set of possible rotation numbers for $L(p,q)$.  We begin with the following lemma.

\begin{lemma}
\label{lem:pq-easy-bound-a1-large}
If $-\frac{p}{q} = [a_1,\dots,a_n]$ and $a_1 \leq -3$, then $p \geq 2q+1$.
\end{lemma}

\begin{proof}
This is obviously true when $n=1$, since $p \geq 3$ and $q=1$.  If $n>1$, then by \eqref{eq:numerator-recursion}, $-\frac{q}{r} = [a_2,\dots,a_n]$ for some $r < q$, and $p = |a_1|q-r \geq |a_1|q - (q-1) = (|a_1|-1)q + 1$.  Since $|a_1| \geq 3$, we have $p \geq 2q+1$ as desired.
\end{proof}

We recall from Section~\ref{ssec:minimal-d3} the notation $d(b_1,\dots,b_k) = |\det(M)|$, where $M$ is the tridiagonal matrix with diagonal entries $b_1,\dots,b_k$ and all entries above and below the diagonal equal to $1$; in particular $d(a_1,\dots,a_n) = p$, and $d()=1$ by convention.

\begin{lemma}
\label{lem:det-product}
Write $-\frac{p}{q} = [a_1,a_2,\dots,a_n]$, $n\geq 2$, and suppose that $a_i \leq -3$ for all $i$.  Then for any $i,j$ with $1\leq i\leq j \leq n$, we have
\[ d(a_1,\dots,a_{i-1})d(a_{j+1},\dots,a_n) < \frac{p}{\prod_{k=i}^j (|a_k|-1)}. \]
\end{lemma}

\begin{proof}
We first prove this in the case $i=j$.  If $i=j=1$ then this amounts to proving that
\begin{equation}\label{eq:d(a2)-bound}
d(a_2,\dots,a_n) < \frac{p}{|a_1|-1}, 
\end{equation}
which follows immediately from noting that $q = d(a_2,\dots,a_n)$ and that if we write $-\frac{q}{r} = [a_2,\dots,a_n]$, then by \eqref{eq:numerator-recursion}, $p = |a_1|q - r > |a_1|q - q$.  (Note that this inequality still holds when $n=1$, in which case it says that $1 < \frac{p}{p-1}$.)  Likewise, for $i=j=n$ we observe that $d(a_1,\dots,a_n) = d(a_n,\dots,a_1)$ and hence 
\[ d(a_1,\dots,a_{n-1}) = d(a_{n-1},\dots,a_1) < \frac{p}{|a_n|-1} \]
by the preceding argument.

If instead we have $1 < i=j < n$, then we recall that $p$ is the order of $H_1(L(p,q))$, where $L(p,q)$ is the result of surgery on a chain of $n$ unknots with framings $a_1,a_2,\dots,a_n$ in order.  If we perform slam dunk operations repeatedly on either end of the chain, until all that remains are the $i$th unknot and one unknot on either side of it, then the framings of the unknots on either end are now $-\frac{r}{s} = [a_{i-1},a_{i-2},\dots,a_1]$ and $-\frac{t}{u} = [a_{i+1},a_{i+2},\dots,a_n]$.  Since Dehn surgery on this 3-component chain produces $L(p,q)$, its first homology is presented by the associated framing matrix, hence
\[ p = \left|\det\left(\begin{array}{ccc} -r & s & 0 \\ 1 & a_i & 1 \\ 0 & u & -t \end{array}\right)\right| 
= \left|a_i rt + ru + st\right|
= rt \left| a_i + \frac{1}{t/u} + \frac{1}{r/s} \right|. \]
Now $\frac{t}{u} > 2$ by Lemma \ref{lem:pq-easy-bound-a1-large} and likewise for $\frac{r}{s}$, so we have
\[ -a_i - \frac{1}{t/u} - \frac{1}{r/s} > |a_i| - 1 > 0 \]
and hence $p > rt(|a_i| - 1)$.  This gives the desired inequality since $r=d(a_1,\dots,a_{i-1})$ and $t=d(a_{i+1},\dots,a_n)$, completing the proof when $i=j$.

Finally, suppose that $i < j$.  Repeated application of \eqref{eq:d(a2)-bound} gives
\[ d(a_{j+1},\dots,a_n) < \frac{d(a_j,\dots,a_n)}{|a_j|-1} < \dots < \frac{d(a_{i+1},\dots,a_n)}{(|a_{i+1}|-1)\dots(|a_j|-1)}, \]
so multiplying both sides by $d(a_1,\dots,a_{i-1})$ produces
\[ d(a_1,\dots,a_{i-1})d(a_{j+1},\dots,a_n) < \frac{d(a_1,\dots,a_{i-1})d(a_{i+1},\dots,a_n)}{\prod_{k=i+1}^j (|a_k|-1)}
< \frac{p}{\prod_{k=i}^j (|a_k|-1)}, \]
where the last inequality follows from applying the case $i=j$ which was already proved.
\end{proof}

We will apply Lemma \ref{lem:det-product} to get an upper bound on $f(\frac{p}{q})$, as computed in \eqref{eq:f-matrix-formula}, which we recall from Section \ref{ssec:recurrence} determines the minimal $d_3$ invariant $d_3(\xican)$ by the formula $d_3(\xican) = \frac{1}{4}\left(-\frac{1}{p}f(\frac{p}{q})+n-2\right)$ where $-\frac{p}{q}$ has continued fraction $[a_1,\dots,a_n]$ of length $n$.  In what follows we will continue to assume that $a_i \leq -3$ for all $i$, though we will only require $n \geq 2$.

More explicitly, recall that $f(\frac{p}{q}) = \rv^T A_{p/q} \rv$, where $\rv = \langle |a_1|-2, \dots, |a_n|-2 \rangle$ and $A_{p/q}$ is the symmetric matrix whose $(i,j)$th entry ($i\leq j)$ is
\[ c'_{ij} = d(a_1,\dots,a_{i-1}) d(a_{j+1},\dots,a_n) \]
according to Proposition~\ref{prop:a-pq-entries}.  We decompose $f(\frac{p}{q}) = \rv^T A_{p/q} \rv$ into two sums (differently than in Section~\ref{ssec:recurrence}).  The first sum comes from the contributions of the diagonal terms in $A_{p/q}$, which satisfy
\begin{equation}\label{eq:Apq-diagonal}
 \sum_{i=1}^n (|a_i|-2)^2 c'_{ii} < \sum_{i=1}^n \frac{p(|a_i|-2)^2}{|a_i|-1} < p \left(\sum_{i=1}^n (|a_i|-2)\right), 
\end{equation}
where we have applied Lemma \ref{lem:det-product} to produce the first inequality.  The other sum comes from the off-diagonal terms, satisfying
\begin{align}
\nonumber 2\sum_{1 \leq i < j \leq n} (|a_i|-2)(|a_j|-2)c'_{ij} &< 2\sum_{i<j} \frac{(|a_i|-2)(|a_j|-2)p}{\prod_{k=i}^{j} (|a_k|-1)} \\
\label{eq:Apq-off-diagonal} &= 2p\sum_{i<j} \frac{1}{\prod_{k=i+1}^{j-1} (|a_k|-1)}\cdot\frac{|a_i|-2}{|a_i|-1}\cdot \frac{|a_j|-2}{|a_j|-1} \\
\nonumber &< 2p\sum_{i<j} \frac{1}{2^{j-i-1}}
\end{align}
again by Lemma \ref{lem:det-product} and the fact that $\frac{1}{|a_k|-1} \leq \frac{1}{2}$.  In the last sum, the quantity $k=j-i-1$ can take any value from $0$ to $n-2$, and each value of $k$ is taken by $n-(k+1)$ pairs $(i,j)$, so we have
\[ \sum_{1 \leq i < j \leq n} \frac{1}{2^{j-i-1}} = \sum_{k=0}^{n-2} \frac{n-1-k}{2^k} = 2n - 4 + \frac{1}{2^{n-2}} \]
as can be shown by an easy induction argument.  In particular, it is bounded above by $2n-3$, so by combining the bounds for the diagonal and off-diagonal terms coming from \eqref{eq:Apq-diagonal} and \eqref{eq:Apq-off-diagonal}, we conclude that
\[ f\left(\frac{p}{q}\right) < p\left(\sum_{i=1}^n(|a_i|-2) + 2(2n-3)\right) = p\left(\sum_{i=1}^n (|a_i|-1) + 3n - 6\right). \]
Combining this with equation \eqref{eq:d3-formula-fpq}, we conclude the following.

\begin{proposition}
\label{prop:f-bound-ai-large}
If $-\frac{p}{q} = [a_1,\dots,a_n]$, with $n \geq 2$ and $a_i \leq -3$ for all $i$, then 
\[ d_3(\xican) > -\frac{1}{4}\left(\sum_{i=1}^n (|a_i|-1) + 2n-4\right), \]
where $\xican$ is the canonical contact structure on $L(p,q)$.
\end{proposition}

Now suppose that $r$ is a rotation number for $L(p,q)$ corresponding to the tight contact structure $\xi$, where $p,q$ are as in Proposition \ref{prop:f-bound-ai-large}.  Then $-\frac{r^2+p}{4p} = d_3(\xi) \geq d_3(\xican)$ by Proposition~\ref{prop:d3-xican-minimal}, so
\begin{equation}
\label{eq:r2-bound-d3-can}
r^2 \leq -p(4d_3(\xican) + 1)
\end{equation}
and Proposition \ref{prop:f-bound-ai-large} ensures that
\begin{equation}
\label{eq:r2-bound-ai-large}
r^2 < p\left(\sum_{i=1}^n (|a_i|-1) + 2n-5\right).
\end{equation}

\begin{proposition}
\label{prop:r-bound-most-n-large}
Suppose that $r$ is a rotation number for $L(p,q)$, where $-\frac{p}{q} = [a_1,\dots,a_n]$ with $n \geq 3$ and $a_i \leq -3$ for all $i$. Then $|r| \leq \frac{p-6}{2}$, except possibly when $\min_i |a_i| = 3$ and $n$ is either $3$ or $4$.
\end{proposition}

\begin{proof}
Let $m\geq 3$ denote the minimum value of $|a_i|$ over all $i$.  Then for any fixed $i$ we have
\begin{align}
\nonumber |a_i|-1 &= \frac{\prod_{j=1}^n (|a_j|-1)}{\prod_{k\neq i} (|a_k|-1)} \\
\label{eq:rot-bounds-product} &\leq \frac{\frac{m-1}{m}(p - (n-1)(m-1)^{n-1})}{(m-1)^{n-1}} \\
\nonumber &= \frac{p}{m(m-1)^{n-2}} - \frac{m-1}{m}(n-1),
\end{align}
since the numerator $\prod_j (|a_j|-1)$ is the number of tight contact structures on $L(p,q)$ by Theorem~\ref{thm:lens-space-classification}, which we bound from above using Proposition \ref{prop:count-xi}.  Combining the inequality \eqref{eq:r2-bound-ai-large} with this bound for each $i$, we get
\begin{align*}
r^2 &< p\left(\left(\frac{n}{m(m-1)^{n-2}}\right)p -\frac{m-1}{m}(n^2 - n) + 2n-5\right) \\
&= p\left(\left(\frac{n}{m(m-1)^{n-2}}\right)p - \frac{m-1}{m}\left(\left(n-\frac{3m-1}{2m-2}\right)^2 + \frac{11m^2-14m-1}{(2m-2)^2}\right)\right)\\
& \leq \left(\frac{n}{m(m-1)^{n-2}}\right)p^2 - \left(\frac{m-1}{m} + \frac{11m^2-14m-1}{4m(m-1)}\right) p \\
& < \left(\frac{n}{m(m-1)^{n-2}}\right)p^2 - \frac{8}{3}p
\end{align*}
since $n-\frac{3m-1}{2m-2} \geq n-2 \geq 1$ and $11m^2-14m-1 > 2\cdot 4m(m-1) > 0$ for all $m \geq 3$.  In particular, if $m \geq 4$ then the coefficient of $p^2$ in this bound is at most $\frac{n}{4\cdot 3^{n-2}} \leq \frac{1}{4}$ for all $n \geq 3$.  Similarly, if $m=3$ and $n\geq 5$ then the coefficient of $p^2$ is $\frac{n}{3\cdot 2^{n-2}} \leq \frac{5}{24}$.  Thus in either case we conclude that $r^2 < \frac{p^2}{4} - \frac{8}{3}p < \left(\frac{p-5}{2}\right)^2$, and since $|r|$ is an integer it must be at most $\frac{p-6}{2}$, as desired.
\end{proof}

In order to deal with the case $m=\min_i |a_i| = 3$ and $n<5$, we can modify the proof of Proposition~\ref{prop:r-bound-most-n-large} by refining the bounds on each $|a_i|-1$ as follows.  Let $M = \max_i |a_i|$, and let $j$ be an index for which $|a_j|=M$.  Then we still have by the same argument as for \eqref{eq:rot-bounds-product} that 
\[ |a_i| - 1 \leq \frac{\frac{2}{3}(p-(n-1)2^{n-1})}{\prod_{k\neq i}(|a_k|-1)}, \]
but now the denominator is at least $(M-1)\cdot 2^{n-2}$ except possibly when $i=j$, in which case it is still at least $2^{n-1}$.  Now the inequality \eqref{eq:r2-bound-ai-large} yields
\begin{equation}
\label{eq:r2-bound-ai-large-refined}
r^2 < p\left( \frac{2}{3}(p-(n-1)2^{n-1})\cdot \left(\frac{n-1}{(M-1) 2^{n-2}} + \frac{1}{2^{n-1}}\right) + 2n-5 \right).
\end{equation}

\begin{lemma}
\label{lem:r-bound-ai-large-n-small}
If $r$ is a rotation number for $L(p,q)$, where $-\frac{p}{q} = [a_1,\dots,a_n]$, $\min_i |a_i| = 3$, and $n$ is either 3 or 4, then $|r| \leq \frac{p-6}{2}$.
\end{lemma}

\begin{proof}
We will establish the bound $r^2 \leq \frac{p^2}{4} - \frac{5}{2}p$.  Indeed, since this is less than $\left(\frac{p-5}{2}\right)^2$ it will follow that $|r| < \frac{p-5}{2}$, and since $|r|$ is an integer we can conclude that $|r| \leq \frac{p-6}{2}$.

We first consider the case $n=4$.  In this case, the inequality \eqref{eq:r2-bound-ai-large-refined} becomes $r^2 < p\left(\frac{2}{3}(p-24)\left(\frac{3}{4(M-1)} + \frac{1}{8}\right) + 3\right)$, or
\begin{equation}\label{eq:r-squared-M}
 r^2 < p\left(\left(\frac{1}{2(M-1)}+\frac{1}{12}\right)p + 1 - \frac{12}{M-1}\right). 
\end{equation}
If $M \geq 4$, we can bound the right hand side of \eqref{eq:r-squared-M} above by $\frac{p^2}{4} - \frac{5}{2}p$ as follows:
\begin{itemize}
\item if $M=4$ then the right hand side of \eqref{eq:r-squared-M} is equal to $\frac{1}{4}p^2 - 3p < \frac{p^2}{4} - \frac{5}{2}p$;
\item if $M=5$ then it is equal to $\frac{5}{24}p^2 - 2p \leq \frac{p^2}{4}-\frac{5}{2}p$, assuming $p \geq 12$;
\item if $6 \leq M \leq 9$ then it is at most $\frac{11}{60}p^2 - \frac{1}{2}p \leq \frac{p^2}{4}-\frac{5}{2}p$, assuming $p\geq 30$;
\item if $M \geq 10$ then it is at most $\frac{5}{36}p^2 +p \leq \frac{p^2}{4}-\frac{5}{2}p$, assuming $p \geq 32$.
\end{itemize}
We can check that in any case, we must have $p \geq 32$ as follows.  If $-\frac{p_i}{q_i} = [a_i,\dots,a_4]$ for $1\leq i\leq 4$, then $p_4 = |a_4| \geq 3$ and $p_i \geq 2p_{i+1}+1$ for $i \leq 3$, by Lemma \ref{lem:pq-easy-bound-a1-large} and the fact that $q_i = p_{i+1}$.  This implies that $p_3 \geq 7$, $p_2 \geq 15$, and $p = p_1 \geq 31$; we could only have equality if $a_i=-3$ for all $i$, but $[-3,-3,-3,-3] = -\frac{55}{21}$, so $p > 31$ after all.  Thus when $n=4$ and $M>3$, we have $r^2 \leq \frac{p^2}{4} - \frac{5}{2} p$, and so $|r| \leq \frac{p-6}{2}$ as claimed.  The remaining case when $n=4$ is $M=m=3$, and we compute for $L(55,21)$ that $d_3(\xican)=-\frac{1}{5} > -\frac{1}{4}$, so $L(55,21)$ has no rotation numbers anyway by Remark~\ref{rmk:large-d3-not-reducible}.

Finally, in the case $n=3$, writing the continued fraction as $[-a,-b,-c]$ for convenience, we compute that $p=abc-a-c$.  Thus equation \eqref{eq:r2-bound-ai-large} says that $r^2 < p(\frac{p}{4}-\frac{5}{2})$ whenever $a+b+c-2 \leq \frac{abc-a-c}{4}-\frac{5}{2}$, or equivalently $5a+4b+5c+2 \leq abc$, and since $b>2$ it suffices to have $5(a+b+c) \leq abc$.  This last inequality is symmetric in $a,b,c$, so we may assume for convenience that $3=a \leq b \leq c=M$; then $5(a + b + c) \leq abc$ is satisfied whenever $b \geq \frac{5M+15}{3M-5}$, and since $b\geq 3$ this is automatic as long as $M \geq \frac{15}{2}$.  We conclude that if $n=3$ and $M \geq 8$, then the rotation numbers satisfy $|r| \leq \frac{p-6}{2}$.

Thus in all cases we conclude that $|r| \leq \frac{p-6}{2}$, except possibly when $-\frac{p}{q} = [a_1,a_2,a_3]$ with $-7 \leq a_i \leq -3$ for all $i$ and $\min_i |a_i|=3$.  There are only 61 such continued fractions, and by insisting that $|a_1| \leq |a_3|$ (since $[a_1,a_2,a_3]$ and $[a_3,a_2,a_1]$ produce homeomorphic lens spaces), we can reduce this number to 35. In each case we explicitly compute $d_3(\xican)$ to show that $-(4d_3(\xican)+1) < \frac{p}{4}-\frac{5}{2}$, and hence that $|r| \leq \frac{p-6}{2}$ via \eqref{eq:r2-bound-d3-can}.  This gives the desired claim.
\end{proof}

\begin{corollary}
\label{cor:p-bound-ai-large}
Let $K$ be a knot with $\maxtb(K) = -\tau$, $\tau > 0$, and define $t=\tau$ if $\tau$ is odd and $t=\tau-1$ if $\tau$ is even.  Suppose that $K$ has a reducible $-p$-surgery for some $p > \tau$, and write $S^3_{-p}(K) = L(p,q) \# Y$.
If the continued fraction $-\frac{p}{q} = [a_1,\dots,a_n]$ satisfies $n\geq 3$ and $a_i \leq -3$ for all $i$, then $p \leq 2t-4$.
\end{corollary}

\begin{proof}
Proposition \ref{prop:r-bound-most-n-large} and Lemma \ref{lem:r-bound-ai-large-n-small} show that the rotation numbers $r$ for $L(p,q)$ satisfy $|r| \leq \frac{p-6}{2}$.  This implies $p-t-1 \leq \frac{p-6}{2}$, or equivalently $p \leq 2t-4$, exactly as in the proof of Proposition \ref{prop:p-q-3-n-2}.
\end{proof}

\bibliographystyle{hplain}
\bibliography{References}

\end{document}